\documentclass{ws-ccm}

\usepackage{mathtools}
\usepackage{graphicx}

\newcommand{\carrenoir}{\nopagebreak\hfill$\blacksquare$\par\medskip}

\newcommand{\eps}{\varepsilon}

\newcommand{\R}{\mathbb R}
\newcommand{\C}{\mathbb{C}}
\newcommand{\NN}{\mathbb{N}}

\newcommand{\HH}{\mathbb{H}}

\newcommand{\beq}{\begin{equation}}
\newcommand{\eeq}{\end{equation}}

\newcommand{\grad}{\nabla}

\newcommand{\wto}{\rightharpoonup}
\newcommand{\trace}[1]{{\rm tr}\left(#1\right)}

\def\leq{\leqslant}
\def\div{{\rm div}\,}

\newcommand{\bdeg}{D}

\newcommand{\DDD}{\mathcal{D}}

\newcommand{\be}{\begin{equation}}
\newcommand{\ee}{\end{equation}}
\newcommand{\nnn}{\nonumber}

\begin{document}

\markboth{S. Alama \& L. Bronsard \& B. Galv\~ao-Sousa} {Weak Anchoring for 2D Liquid Crystals}

\title{Weak Anchoring for a Two-Dimensional Liquid Crystal}

\author{STAN ALAMA${}^1$ \& LIA BRONSARD${}^1$ \& BERNARDO GALV\~AO-SOUSA${}^2$}

\address{${}^1$Department of Mathematics and Statistics, McMaster University,
Hamilton, ON, L8S 4K1, Canada
\\
${}^2$Department of Mathematics, University of Toronto, Toronto, ON, M5S 2E4, Canada\\
\emph{\texttt{alama@mcmaster.ca \quad bronsard@mcmaster.ca \quad beni@math.toronto.edu}}}

\maketitle

\centerline{\bf\today}

\vskip 0.25truein

\begin{abstract}
{\bf Abstract.} 
We study the weak anchoring condition for nematic liquid crystals in the context of the Landau-De~Gennes model.  We restrict our attention to two dimensional samples and to nematic director fields lying in the plane, for which the Landau-De~Gennes energy reduces to the Ginzburg--Landau functional, and the weak anchoring condition is realized via a penalized boundary term in the energy. We study the singular limit as the length scale parameter $\eps\to 0$, assuming the weak anchoring parameter $\lambda=\lambda(\eps)\to\infty$ at a prescribed rate.  We also consider a specific example of a bulk nematic liquid crystal with an included oil droplet and derive a precise description of the defect locations for this situation, for $\lambda(\eps)=K\eps^{-\alpha}$ with $\alpha\in (0,1]$.  We show that defects lie on the weak anchoring boundary for $\alpha\in (0,\frac12)$, or for $\alpha=\frac12$ and $K$ small, but they occur inside the bulk domain $\Omega$ for $\alpha>\frac12$ or $\alpha=\frac12$ with $K$ large.

\end{abstract}


\ccode{Mathematics Subject Classification 2000: }

\section{Introduction}

In this paper we examine the weak anchoring condition for nematic liquid crystals in the context of the Landau-De~Gennes model.  Weak anchoring refers to the imposition of boundary behavior  by means of energy penalization, rather than via a nonhomogeneous Dirichlet condition (which is referred to as ``strong anchoring''.)  We restrict our attention to two--dimensional samples and to nematic director fields lying in the plane.  With this dimensional restriction, the Landau-De~Gennes energy reduces to the familiar Ginzburg--Landau energy, for a complex valued order parameter $u$ which is mapped to the Q-tensor in the Landau-De~Gennes theory, and the weak coupling condition is expressed as a boundary penalization term added to the Ginzburg--Landau energy. We study the singular limit as the length scale parameter $\eps\to 0$, assuming the weak anchoring penalization strength $\lambda=\lambda(\eps)\to\infty$ at a prescribed rate.  We also consider a specific example of a bulk nematic liquid crystal with an included oil droplet \cite{KL}, and derive a precise description of the defect locations for this situation,  depending on the relative strength of the weak anchoring parameter $\lambda(\eps)$.  Although the Ginzburg--Landau functional represents a highly simplified model for nematic liquid crystals, we expect that it nevertheless captures the salient information concerning the formation of singularities under the weak anchoring condition.

We first describe our results in the context of the Ginzburg--Landau model with boundary penalization; the description of the Landau-De~Gennes model and the physical droplet setting, together with the reduction to the Ginzburg--Landau energy, will be explained afterwards.  In particular, the solution to the droplet problem is stated in Theorem~\ref{LQthm} below.
Let 
$$\lambda=\lambda(\eps)=K\eps^{-\alpha}$$ 
for $\alpha\in (0,1]$, $K>0$ constant.  
We impose the weak anchoring condition on a connected component $\Gamma$ of $\partial\Omega$ via a boundary term in the energy.  Let $g: \ \Gamma\to S^1$ be a $C^2$ smooth map, and define
$$
E_{\eps}(u) 
	:= \frac{1}{2} \int_{\Omega} \left( |\grad u|^2 + \frac{1}{2\eps^2} \big( |u|^2-1\big)^2\right)  \, dx 
	  + \frac{\lambda}{ 2} \int_{\Gamma} |u-g|^2 \,dS
$$
A critical point of $E_\eps(u)$ in $H^1(\Omega;\C)$ solves
\begin{equation}\label{EL}
\left.
\begin{gathered}
-\Delta u + \frac{1}{\eps^2}(|u|^2-1)u =0, \quad\text{in $\Omega$},\\
\frac{\partial u}{\partial \nu} + \lambda (u-g)=0, \quad\text{on $\Gamma$}.
\end{gathered}\right\}
\end{equation}
We consider three different geometries, each with some physical motivation.  

\noindent {\bf Problem I:} \
$\Omega\subset\R^2$ is simply connected and with smooth $C^2$ boundary $\partial\Omega=\Gamma$.  In this case, the appropriate space is
$\HH_I:= H^1(\Omega;\C)$, and \eqref{EL} gives the Euler-Lagrange equations corresponding to this variational problem.

\smallskip

\noindent {\bf Problem II:} \ 
$\Omega = \Omega_1\setminus\overline{\Omega_0}$ is a topological annulus, with $C^2$ smooth boundary in two components, $\Gamma=\partial\Omega_0$ the interior boundary, and $\partial\Omega_1$ the exterior.  We impose weak anchoring via $g: \ \Gamma\to S^1$ on the interior boundary, and a constant Dirichlet condition on the exterior, so the Euler-Lagrange equations are \eqref{EL} with the additional condition,
\be\label{ELII}
u=1, \quad\text{on $\partial\Omega_1$}.
\ee
The appropriate space is 
$$ \HH_{II}:= \{ u\in  H^1(\Omega;\C): \ u=1 \ \text{on $\partial\Omega_1$}\}.  $$

The choice of a constant as a Dirichlet (strong anchoring) boundary conditon is motivated by the physical model of a droplet $\Omega_0$ included in a bulk nematic (described below); mathematically, the problem may be posed with any $S^1$-valued map imposed on the outer boundary $\partial\Omega_1$.

\noindent {\bf Problem III:} \
$\Omega=\R^2\setminus\Omega_0$ is an exterior domain, with boundary $\Gamma=\partial\Omega_0$.  We impose a weak anchoring condition on $\Gamma$ via the $C^2$ map $g:\ \Gamma\to S^1\subset\mathbb{C}$, and assume that there exists a constant $\phi_0\in (-\pi,\pi]$ for which
\begin{equation}\label{ELIII}
u(x)\to e^{i\phi_0} \quad\text{as $|x|\to\infty$}.
\end{equation}
We minimize $E_\eps$ in the space
$$  \HH_{III}:= \{ u\in  H_{loc}^1(\Omega;\C): \ \exists\phi_0\in\R \ \text{such that $u\to e^{i\phi_0}$ as $|x|\to\infty$}\},
$$
and 
minimizers satisfy the Euler-Lagrange equations \eqref{EL} in the unbounded domain $\Omega$, with asymptotic condition \eqref{ELIII}.  As in Problem II, the choice of a constant at infinity is motivated by the droplet problem posed in \cite{KL}.

The space $\HH_{III}$ is problematic, as the Dirichlet energy does not control the phase of $u$ as $|x|\to\infty$, and in fact the existence of minimizers for fixed $\eps>0$ is not immediate.  Indeed, unlike the Dirichlet problems I and II, we may not specify a limiting constant as $|x|\to\infty$; the asymptotic phase $\phi_0$ is an unknown in the problem, determined by the choice of $\Omega_0$ and $g$.  In the application to nematic liquid crystals, $\Omega_0=D_1(0)$ a disk, and $g=e^{iD\theta}$ is symmetric, and in this case we may in fact conclude that the energy minimizers satisfy $u(x)\to 1$ as $|x|\to\infty$ (see Theorem~\ref{exterior}.)

Our aim in this paper is to study the minimizers of $E_\eps$ as $\eps\to 0$, for each problem I, II, III, and determine how the location of the vortices is affected by the weak anchoring strength $\lambda=\lambda(\eps)=K\eps^{-\alpha}$.  In particular, we observe that $\alpha=\frac12$ is the critical value for the weak anchoring strength, with vortices lying on the boundary component $\Gamma$ for $\alpha<\frac12$ and inside $\Omega$ for $\alpha>\frac12$.
Here is our main result for Problems I, II, and III:

\begin{theorem}\label{mainthm}
Let $g: \ \Gamma\to S^1$ be a given $C^2$ function with degree $\mathcal{D}\in\NN$.
Let $u_\eps$ be minimizers of $E_\eps$ in one of the spaces $\HH_i$, $i=I, II, III$.  For any sequence of $\eps\to 0$ there is a subsequence $\eps_n\to 0$ and $\mathcal{D}$ points $\{p_1,\dots,p_{\mathcal{D}}\}$ in $\Omega\cup\Gamma$ such that
$$  u_{\eps_n}\to u_* \ \text{in} \  C^{1,\mu}_{loc}(\overline{\Omega}\setminus \{p_1,\dots,p_{\mathcal{D}}\}), $$
for $0<\mu<1$, with $u_*: \ \Omega\setminus \{p_1,\dots,p_{\mathcal{D}}\}\to S^1$ a harmonic map.
Moreover, 
\begin{enumerate}
\item[(a)]  $u_*=g$ on $\Gamma\setminus \{p_1,\dots,p_{\mathcal{D}}\}$.  
\item[(b)]  For each $i=1,\dots,\mathcal{D}$, $\deg(u_*;p_i)=1$ in problem I, and  $\deg(u_*;p_i)=-1$ in problems II and III.
\item[(c)]  If $0<\alpha<\frac12$, each $p_i\in\Gamma$; if $\frac12<\alpha\le 1$, then $p_i\in\Omega$ for all $i=1,\dots,\mathcal{D}$.
\item[(d)]  If $\alpha=\frac12$, there exist $K_0<K_1\in\R$ such that the vortices lie on $\Gamma$ for $K<K_0$ and they lie inside $\Omega$ for $K>K_1$.
\item[(e)]  There are Renormalized Energy functions $W_\Omega: \ \Omega^{\mathcal{D}} \to \R$ and $W_\Gamma: \ \Gamma^{\mathcal{D}}\to \R$
such that if $(p_1,\dots,p_{\mathcal{D}})$ lie on $\Gamma$, they minimize $W_\Gamma$, and if they lie inside $\Omega$ they minimize $W_\Omega$.
\end{enumerate}
\end{theorem}
The Renormalized Energies will be defined and their properties analyzed in section~\ref{RNsec}.
The passage to the limit in Theorem~\ref{mainthm} is done using $\eta$-compactness (or $\eta$-ellipticity) methods, introduced by Struwe \cite{Struwe}, Rivi\`ere \cite{Riviere}, and the Renormalized Energy analysis follows the treatment of the Dirichlet problem by Bethuel-Br\'ezis-H\'elein \cite{BBH2}.  The boundary vortices may be treated in a similar way as in thin-film models of micromagnetics, as analyzed by Kurzke \cite{Ku} and Moser \cite{Moser}, although the boundary condition itself is not the same.  Similar estimates (although for a very different problem) were employed by Andr\'e and Shafrir \cite{ASh}.

It is for Problem III that we obtain our most complete results, and it is this case (with interior boundary $\Gamma=\partial B_1(0)$ and $g=e^{i\theta}$) which is directly motivated by physical considerations.  These are described together with the physical context in the following paragraphs, and in Theorem~\ref{LQthm}.

\medskip

\paragraph{Models of Nematic Liquid Crystals}

The equilibrium state of a nematic liquid crystal (in dimension $N$, $N=2,3$,) may be described by a unit {\em director field} $n(x)$, $|n(x)|=1$ at each $x\in W\subset\R^N$.  An early (and widely used) simplified model for nematics is the Oseen-Frank model \cite{Ericksen, HKL}, in which the director is taken to be an $S^{N-1}$-valued vector field, $n: \ W\subset\R^N\to S^{N-1}$.  Assuming all elastic constants to be equal, the 
director minimizes the Dirichlet energy, and thus is a harmonic map with values in $S^{N-1}$.  

An objection to the Oseen-Frank approach is that the director $n(x)$ is a vector field, and hence carries an orientation at each point, whereas the directors $n(x)$ and $-n(x)$ represent the same physical state of the nematic liquid crystal at $x$.  A more appropriate description of the nematic would entail a field taking values in the projective plane $\R P^{N-1}$, not the sphere.  De Gennes proposed a mechanism to represent non-oriented direction fields by means of a symmetric trace-zero N by N matrix-valued function $Q(x)$, called a {\em Q-tensor}.  The class of all nematic directors $n(x)$, $|n(x)|=1$ with the identification $n\sim -n$ is embedded as a subspace in the linear space of traceless symmetric matrices via $Q(x)= s(n\times n -\frac1N Id)$, where $s$ is a scalar.  The Q-tensors which are associated to unit director fields in this way are called {\em uniaxial}.

The Landau-de Gennes functional measures the Dirichlet energy of a Q-tensor while penalizing tensors which are not uniaxial \cite{KL, BZ1, BZ2, BPP,MN}:
$$
\mathcal{F}_{LdG}(Q) := \int_\Omega \left(\frac12 |\grad Q|^2 + \frac1L f_B(Q) \right)\, dx,
$$
with
$$
f_B(Q) := -\frac{a}{2} \trace{Q^2} - \frac{b}{3} \trace{Q^3} + \frac{c}{4} \big( \trace{Q^2}\big)^2 - d,
$$
with (temperature dependent) constants $a, b, c$; the constant $d$ may be chosen so that $\min f_B=0$.  Assuming that the temperature is below the critical temperature for the nematic to isotropic transition, we take the values of $a,b,c>0$.
Then $f_B$ is minimized for uniaxial $Q$, of the form 
\be\label{uniax}
Q=s_+\left(n\otimes n - \frac1N Id\right),
\ee
 with a specific constant $s_+=s_+(a,b,c)>0$.  When $N=3$,  $s_+ = \frac{b + \sqrt{b^2 + 24 a c}}{4c}$, 
and for $N=2$, $s_+=\frac{a\sqrt{2}}{ c}$ (see \cite{Majumdar}.)
For such uniaxial $Q$,   the Landau-de Gennes functional reduces to a constant multiple of the Dirichlet energy of $n$.  Thus, $\mathcal{F}_{LdG}$ is a relaxation of the harmonic map energy of uniaxial tensor fields, in the same way that the Ginzburg-Landau model is for harmonic maps to $S^n$.
As is observed in \cite{BZ1}, for many problems involving singularities in nematic liquid crystals the energy minimizing director field may not be representable by orientable $n(x)$, and thus the Oseen-Frank model cannot always determine the optimal configuration in these examples.  As above, we write the Landau-de Gennes functional assuming the equality of the elastic constants (splay, twist, and bend); a more accurate model would have an anisotropic gradient energy with separate terms for each elastic distortion of the crystal.

In this paper we restrict our attention to planar (thin film or cylindrical) samples, for which the director lies in the same plane as the sample. 
In the non-oriented (projective) case, there are two settings in which planar Q-tensors lead to a Landau-de Gennes model which is equivalent to the Ginzburg-Landau energy.  In the first setting \cite{Majumdar}, we consider the space $\mathcal{Q}_2$ of $ 2\times 2$ traceless symmetric matrices.  Elements of $\mathcal{Q}_2$ are parametrized by two real coordinates, and so the space may be associated with $\C$.  In addition,
 the potential $f_B$ is then minimized on the set of uniaxial tensors of the form
$$  Q=\frac{a\sqrt{2}}{ c} \left( n\otimes n -\frac12 Id\right).  $$
Following \cite{Majumdar}, the energy $\mathcal F_{LdG}$ may be exactly transformed to the Ginzburg-Landau model via the order parameter defined by $u=\frac{2}{ s_+}\left[q_{11}+ iq_{12}\right]$.  We note that if $n=e^{i\phi}$, the corresponding uniaxial Q-tensor is
$$  Q=\frac{a}{ c\sqrt{2}}\begin{pmatrix}  \cos (2\phi) & \ \sin (2\phi) \\
   \sin(2\phi) & -\cos(2\phi)
\end{pmatrix},
$$
and so the associated complex order parameter has a doubled phase, $u=e^{2i\phi}$.  Thus, a simple vortex in the Ginzburg-Landau representation yields a non-orientable half-degree singularity in the associated Q-tensor (see Figure~\ref{fig:vortices}).

\begin{figure}[!htbp]
\begin{center}
\begin{tabular}{cccc}
\includegraphics*[width=100pt]{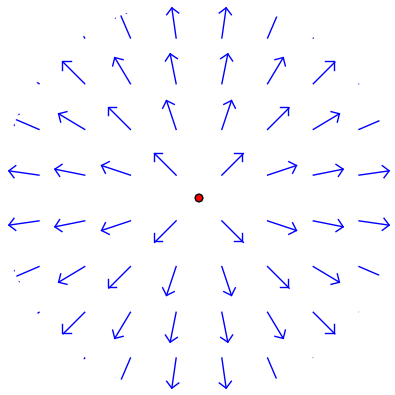}
	&\includegraphics*[width=100pt]{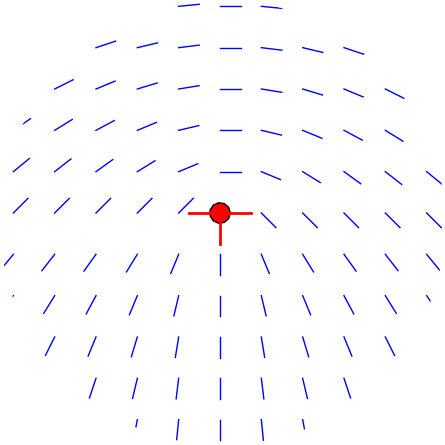}
& \includegraphics*[width=100pt]{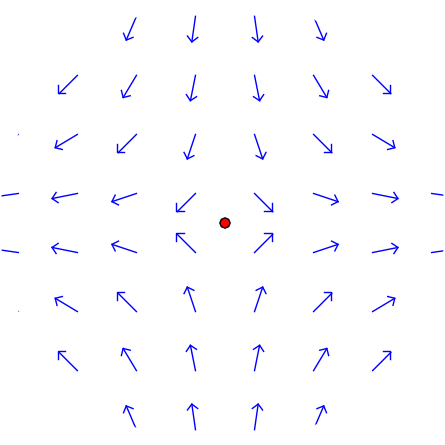}
	& \includegraphics*[width=100pt]{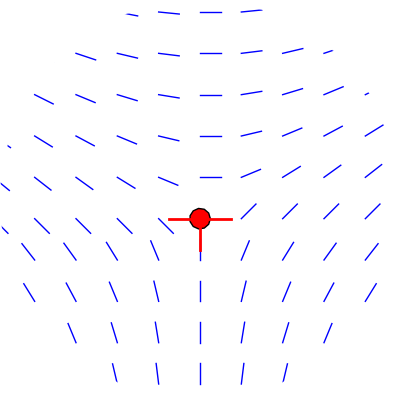} \\
(a) & (b) 
& (c) & (d) \\
\end{tabular}
\end{center}
\caption{Some sample defects:  (a) Oriented degree $+1$ vortex; (b) Non-oriented degree $+\frac12$ defect; (c) Oriented degree $-1$ vortex; (d) Non-oriented degree $-\frac12$ defect.}\label{fig:vortices}
\end{figure}

A different representation of planar Q-tensors may be derived as in \cite{BZ1}, using three-dimensional symmetric traceless matrices but restricting to uniaxial configurations \eqref{uniax} corresponding to planar $n=(n_1,n_2,0)$.  For such planar $n$, the uniaxial Q-tensors may be represented by means of an order parameter (or auxiliary vector field, see \cite{BZ1},) 
$$ u= \frac{2}{ s}Q_{11}-\frac12 +i\frac{2}{ s} Q_{12} = 
   2n_1^2 -1 + 2in_1n_2.  $$
For any $n\in S^1$ we may thus determine a unique $u$ with $|u|=1$, and inversely for $S^1$-valued $u$ we may recover a unit vector $n$ (modulo $n\sim -n$) via a unique uniaxial Q-tensor,
\begin{equation}
Q = s_+\begin{pmatrix}
n_1^2-\frac13  & \quad & n_1n_2 & 0 \\
n_1n_2 & & n_2^2-\frac13 & 0 \\
0 & & 0 & -\frac13
\end{pmatrix} 
= \frac{s_+}{2} \begin{pmatrix}
u_{1}+\frac13 & u_{2} & 0 \\
u_{2} & \frac13 - u_1 & 0 \\
0 & 0 & -\frac23
\end{pmatrix}.
\label{2DQ}
\end{equation}
%
%
It may then be shown \cite{BZ1} that the Landau-de Gennes energy for $Q$ of the form \eqref{2DQ} reduces to a constant multiple of a Ginzburg-Landau energy for $u$.

We note that this procedure of reducing the Landau-de Gennes model for planar uniaxial Q-tensors in three dimensions to the classical Ginzburg-Landau model is not an equivalence.  Indeed, as has been noted in \cite{BPP}, a more complete representation of planar Q-tensors involves both a complex order parameter $u$ and a scalar function $s=s(x)$, giving rise to a more complex planar system with three real unknown functions. 
Nevertheless, we expect that the results concerning the strength of the weak anchoring constant and the formation of defects obtained in the Ginzburg--Landau setting of this paper will extend to the more refined models (as in \cite{BPP},) as the energy costs associated to boundary and interior vortices will be of the same order of magnitude in both the simpler and more refined models. 

\medskip 
  
 As we will see, non-orientability will be an essential feature of minimizers in two dimensions.  However, for comparison, we point out that the Ginzburg-Landau energy may be used as a very simple model for oriented directors, as it is a relaxation of the the $S^1$ harmonic map energy.  The complex order parameter $u: \ \Omega\subset \R^2\to \C$, but the constraint $u\in S^1$ is obtained by the penalization term in the energy.  The singularities in the liquid crystal will correspond to regions where $|u|\ll 1$, and thus disobeys the $S^1$ constraint.  This is a very simplified model of liquid crystals with planar directors, and leads to the energy functional $E_\eps(u)$ for the case of orientable 2D director fields with weak anchoring.  Although this model is very simple, it serves to illustrate the importance of nonorientability in the study of defects in 2D (see Remark \ref{OF}.) We again note that a more realistic model of nematics is anisotropic, due to different values of the elastic coefficients in the gradient term, as in the widely accepted Ericksen model \cite{Ericksen}.  The effect of anisotropy in two dimensional liquid crystals has been recently studied in \cite{CKP}.

\paragraph{Weak Anchoring}
Following \cite{MN}, the weak anchoring condition is obtained by introducing a surface term in the energy
$$   \mathcal{F}_\Gamma = \frac{W}{ 2}\int_\Gamma \left(Q-Q_\Gamma\right)^2 ds,  $$
where $Q_\Gamma$ is the value of the (uniaxial) Q-tensor preferred by the boundary material $\Gamma$, and $W>0$ is a constant giving the anchoring energy along $\Gamma$.  By introducing the complex order parameter $u$ as above (either interpreting $Q\in \mathcal{Q}_2$ as a two-dimensional traceless symmetric matrix as in \cite{Majumdar} or by the ansatz \eqref{2DQ} as in \cite{BZ1},) this translates into a similar penalization term to be added to the classical Ginzburg-Landau energy for the order parameter,
$\frac{\tilde W}{ 2}\int_\Gamma |u-g|^2\, ds$, where $g:\Gamma\to S^1$ is the order parameter associated to the given tensor $Q_\Gamma$.  Thus, after nondimensionalization,we obtain the energy $E_\eps$ for the order parameter $u$ subjected to a weak anchoring condition on $\Gamma$.

\paragraph{Weak Anchoring Around a Droplet}
In a nematic, it is common to assume {\em homeotropic} anchoring, in which the preferred direction is with the director $n$ aligned along the unit normal $\nu$ to $\Gamma$ at each point.
As noted above, if we represent $\nu=e^{i\phi(s)}$ in complex notation, with $\Gamma$ parametrized by arclength $s$, the uniaxial Q-tensor associated to $\nu$ will have complex order parameter $u=e^{2i\phi(s)}$.  In particular, for a simple closed boundary component $\Gamma$, the  normal field $\nu$ being of degree one, we will thus obtain an order parameter with $\deg(u;\Gamma)=2$.  As it is well known (see \cite{BBH2}) that for small $\eps>0$, interior vortices for Ginzburg-Landau minimizers must be of degree $\pm 1$, this implies that minimizers of Landau-de Gennes (under the above planar ansatz) will prefer pairs of non-orientable half-degree singularities rather than ``hedgehog'' shaped degree-one vortices.

Following an example in \cite{KL}, we consider the case of a bulk nematic liquid crystal with an included oil droplet.  In our two-dimensional setting, the oil droplet is assumed to be circular, and the nematic occupies the exterior domain, which we assume is either a large disk (Problem II) or the entire plane excluding the droplet (Problem III.)  In either case, we assume that the droplet is of unit radius, and centered at the origin, and so the homeotropic weak anchoring condition prefers a director $n=\nu=e^{i\theta}$, written in complex notation.  As observed above, this corresponds to the choice
$$   g(\theta)= e^{2i\theta},  $$
of degree $\bdeg=2$ in Theorem~\ref{mainthm}.
As a corollary of Theorem~\ref{mainthm} and the detailed study of the associated Renormalized Energies (in section~\ref{RNsec}) we have:
\begin{theorem}\label{LQthm}
Let $\Omega=\R^2\setminus B_1(0)$,  $g(\theta)=e^{2i\theta}$, $0\le\theta<2\pi$, and $u_\eps$ the minimizers of $E_\eps$ in $\HH_{III}$ corresponding to $\Omega$ and $g$.
Then, there exist points $p_1=(0, t)$, $p_2=(0, -t)$, with $t\ge 1$ such that $u_\eps\to u_*$ in $C^{k,\mu}_{loc}(\overline{\Omega}\setminus\{p_1,p_2\})$, with $u_*$ an $S^1$-valued harmonic map, and any $k\ge 0$.  Moreover,
$u_*\to 1$ as $|x|\to\infty$, $\deg(u_*, p_j)=-1$, and 
\begin{enumerate}
\item[(1)] If $0<\alpha<\frac12$, both antivortices lie on $\partial B_1(0)$, $p_1=(0,1)$, $p_2=(0,-1)$.
\item[(2)] If $\frac12 <\alpha\le 1$, both antivortices lie inside $\Omega$, $p_1=(0,\sqrt[4]{2})$, $p_2=(0,-\sqrt[4]{2})$.
\item[(3)] If $\alpha=\frac12$, there exists $K_0\le K_1$ such that both antivortices lie on $\partial B_1(0)$ for $K<K_0$ and inside $\Omega$ for $K>K_1$.
\end{enumerate}
\end{theorem}

We recall that a degree $\pm 1$ vortex for $u$ corresponds to a half-vortex for the associated director $n$. 
The conclusions of the theorem are illustrated in Figure~\ref{fig:bdy-int-vortices}.

\begin{figure}[!htbp]
\begin{center}
\begin{tabular}{ccc}
\includegraphics*[width=200pt]{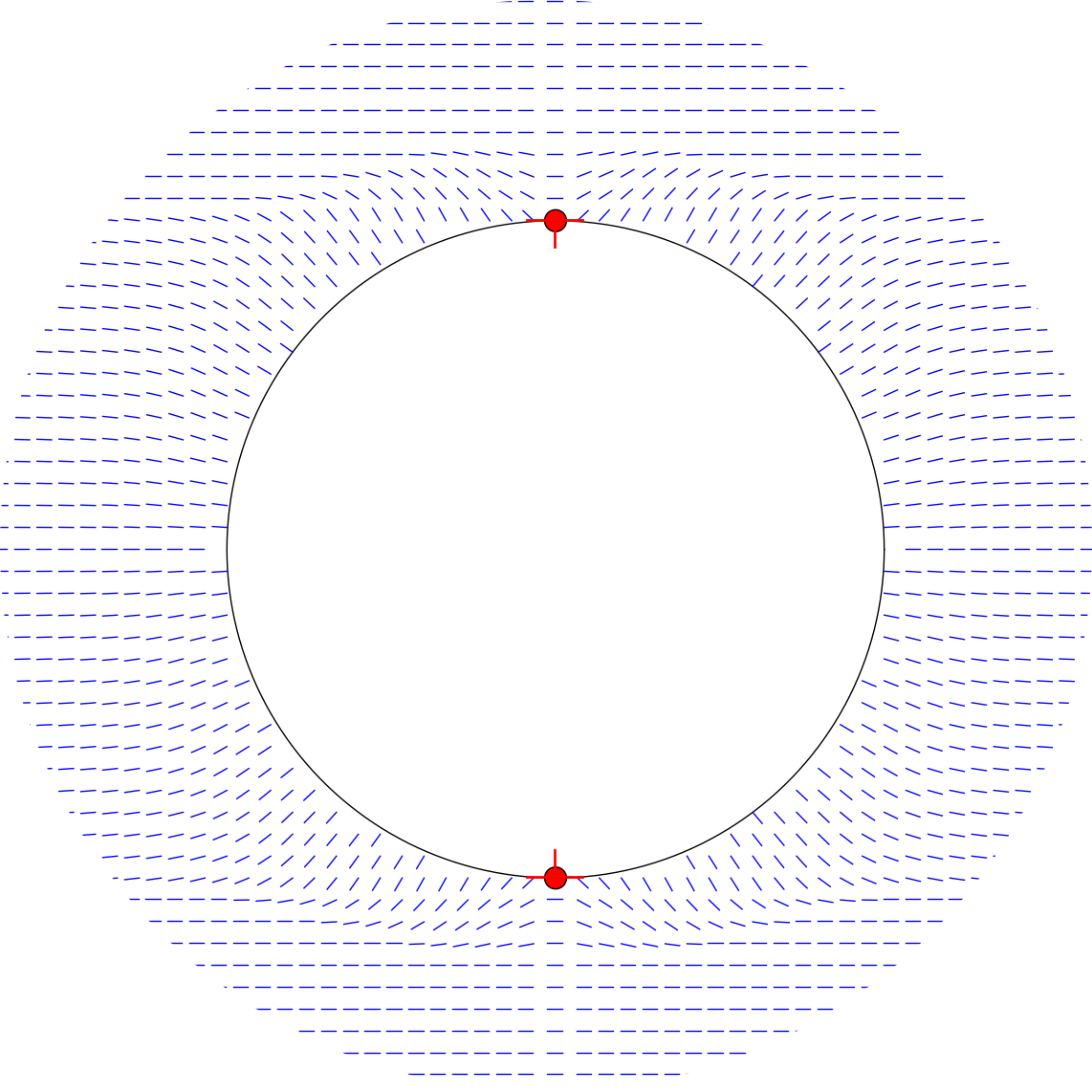}
	& \qquad\qquad & \includegraphics*[width=200pt]{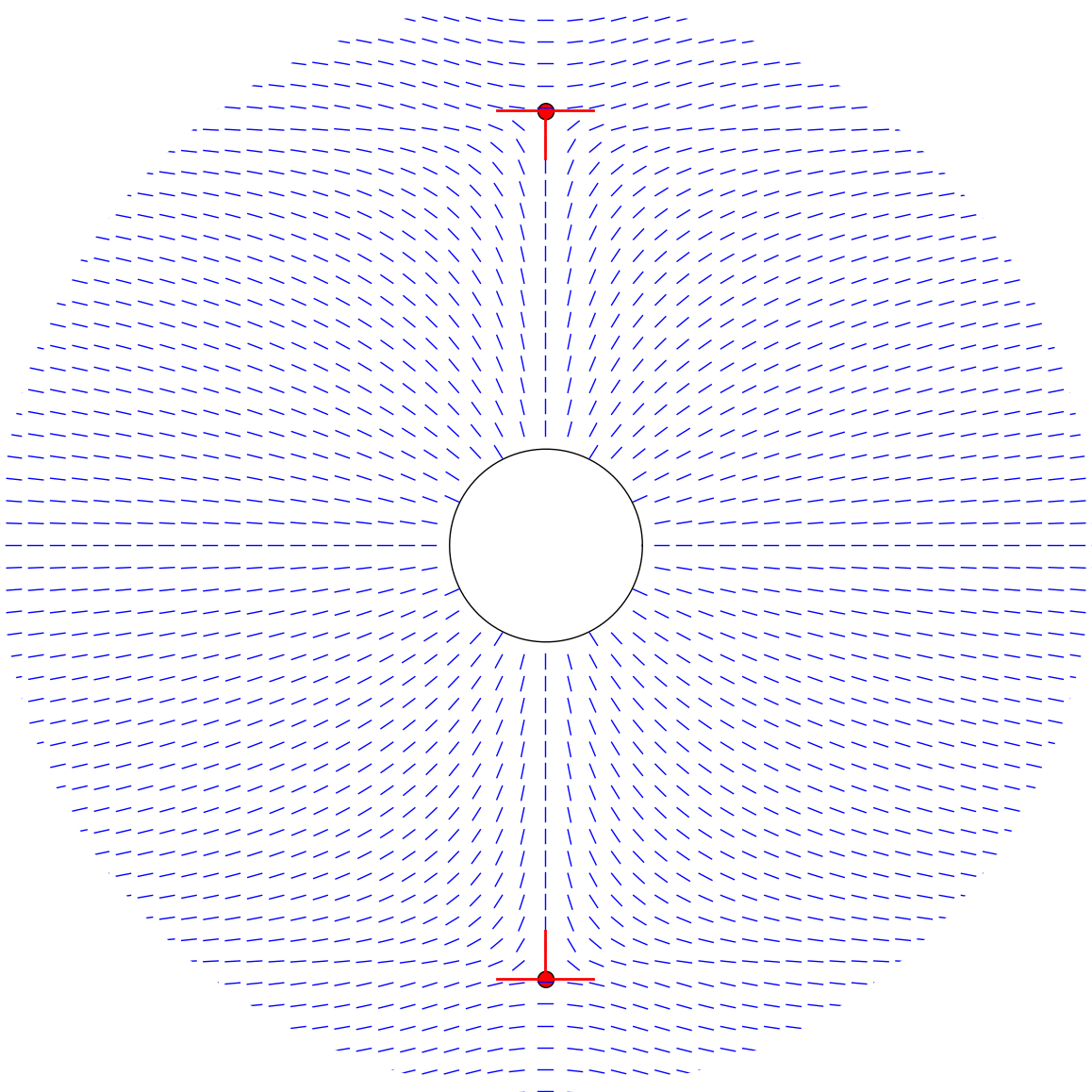} \\
(a) & & (b) 
\end{tabular}
\end{center}
\caption{(a) Boundary vortex for $0<\alpha<\frac12$; (b) Interior vortex for $\frac12 < \alpha\leq1$.}\label{fig:bdy-int-vortices}
\end{figure}

We observe that $\alpha=\frac12$ is critical for the scaling in this problem.  In particular, if we consider minimizing
$$  \tilde E_{\eps,R}(v) = \int_{\R^2\setminus B_R(0)} \left[
 \frac12 |\nabla v|^2 + \frac{1}{ 4\eps^2} (|v|^2-1)^2\right] dx
   + \frac{1}{ 2\eps^{1/2}}\int_{\partial B_R(0)} |v-g|^2\, ds, $$
with $\tilde\Omega_R=\R^2\setminus \overline{B_R(0)}$, $v\in\HH_{III}$, then by rescaling $v(x)=v(Ry)=u(y)$, $|y|>1$, we obtain
$\tilde E_{\eps,R}(v)= E_{\eps/R}(u)$, with $K=\sqrt{R}$.  Thus, with critical $\alpha=\frac12$, minimizers in the exterior of droplets of large radius $R$ will have pairs of half anti-vortices lying in the exterior domain $\tilde\Omega_R$, whereas for droplets of small radius $R$ the vortices will cling to the boundary $\Gamma$.  This conclusion is very different from that drawn in \cite{KL}, which predicts a single hyperbolic (degree $-1$) vortex along the axis of symmetry for large (3D) spherical droplets.  On the other hand, the result we obtain here is consistent with a two-dimensional cross-section of the ``Saturn ring'' configuration predicted for smaller sized droplets in \cite{KL} (see Figure~7 of that paper.)  The difference with \cite{KL} is due to the two-dimensional geometry of our problem.  Indeed, our ``point'' disclinations are in fact line singularities in a cylindrical three-dimensional setting, whereas the hyperbolic hedgehog found in \cite{KL} is a true point defect.  The calculation of the energy of each singularity is thus different in different dimensions.  In particular, in 3D the half-degree disclinations are line singularities, forming loops (as for the Saturn rings,) and will be energetically favorable only if the length of the disclination loop is small.

\begin{remark}\label{OF}
If we were to restrict our attention to oriented director fields $n(x): \ \Omega\to S^1$, using the Ginzburg--Landau energy $E_\eps$ as a relaxation of the harmonic map energy, Theorem~\ref{mainthm} implies a very different form for minimizers.  In this orientable Oseen-Frank setting, the homeotropic anchoring condition imposes $g(\theta)=e^{i\theta}$ on $\Gamma=\partial B_1(0)$.  In this case $\DDD=1$, and there is a single antivortex $p\in \overline\Omega$, with all of the conclusions as in Theorem~\ref{mainthm}.  The explicit form of the Renormalized Energy in this case predicts a single, (orientable) degree -1 antivortex, behind the droplet:  we have $p=(-1,0)\in\Gamma$ for $\alpha<\frac12$ (or $\alpha=\frac12$ and $K$ small), and $p=(-2,0)\in\Omega$ for $\alpha>\frac12$ (or $\alpha=\frac12$ and $K$ large.) This illustrates the importance of orientability in the analysis of the physical liquid crystal problem. 
\end{remark} 

\smallskip

\paragraph{Micromagnetics}
We remark that the mechanism of imposing boundary behavior via energy penalization is also present in other physical contexts.  Notable among these are models of thin film micromagnets (see \cite{DKMO}.)  For these energies, similar analyses exploiting the connection to the Ginzburg-Landau functional have been undertaken by Kurzke \cite{Ku} and Moser \cite{Moser}.  There are two essential differences between the micromagnetic models and Landau-De~Gennes:  the first is that magnetic materials do have an oriented, $S^2$-valued magnetization vector.  The second is the physics of the boundary behavior, as the magnetization vector tends to point {\em tangentially} to any boundary component, not homeotropically (as a nematic.)  As we will see in our analysis of the singular limit $\eps\to 0$, this difference is reflected in the cost of boundary vortices, and the critical weak coupling will occur at $\alpha=1$ rather than our $\alpha=\frac12$ as a result.  Nevertheless, the methods derived in \cite{Ku, Moser} will be very useful in the analysis of the energy $E_\eps$.

\section{The exterior domain}

For fixed $\eps, \lambda$, the existence of a minimizer in Problems I and II follows from standard arguments.   Problem III, posed in the exterior domain $\Omega=\R^2\setminus\Omega_0$,  requires some more care, and we present here an existence result for minimizers.

For $\omega\subset\Omega$, we define a localized energy,
$$
E_{\eps}(u;\omega) 
	:= \frac{1}{2} \int_{\omega} \left( |\grad u|^2 + \frac{1}{2\eps^2} \big( |u|^2-1\big)^2\right)  \, dx 
	+ \frac{\lambda}{ 2} \int_{\Gamma\cap\overline{\omega}} |u-g|^2 \,dS.
$$
  We also define some useful spaces,
\begin{gather*}  X:=\{ u\in H^1_{loc}(\R^2\setminus\Omega_0): \ \exists\phi_0\in\R \ \text{such that $u(x)\to e^{i\phi_0}$ as $|x|\to\infty$}\}.  
\\
X_0:=\{ u\in X: \ \text{$u(x)\to 1$ as $|x|\to\infty$}\},
\\ 
X_{\phi,R}:= \{ u\in H^1(B_R\setminus\Omega_0) \ : \ \text{$u(x)=e^{i\phi}$ on $\partial B_R$}\},
\end{gather*}
and consider minimization of $E_\eps$ in each class,
$$   m:=\inf_{u\in X} E_\eps(u), \qquad  m_0:=\inf_{u\in X_0} E_\eps(u),\qquad m_{\phi,R}:=\inf_{u\in X_{\phi,R}} E_\eps(u; B_R\setminus\Omega_0).
$$

\begin{theorem}\label{exterior}
Let $\Omega_0\subset\R^2$ be a bounded, smooth, simply connected domain, and $\Omega=\R^2\setminus\Omega_0$.  Then, for each fixed $\eps>0$, $m=\min_X E_\eps$ is attained, by a solution of \eqref{EL} with \eqref{ELIII} holding for some $\phi_0\in\R$.  If $\Omega_0=B_{R_0}$ is a disk and $g=g(\theta)=e^{iD\theta}$, then $m_0=\min_{X_0} E_\eps$ is also attained (with $\phi_0=0$.)
\end{theorem}

\begin{proof}[]
First, by standard arguments in the calculus of variations, $m_{0,R}$ is attained for all $R>\text{diam}\,(\Omega_0)$, by a solution $u_R(x)$ of 
\eqref{EL} with \eqref{ELII} on $\partial\Omega_1=\partial B_R$.  By 
Lemma~\ref{apriori}, $|u_R(x)|\le 1$ and there exists a constant $C$, independent of $R$, for which $|\nabla u_R|\le C/\eps$.  By standard elliptic estimates and a diagonal argument, there exists a subsequence $R_j\to\infty$ and $u\in C^k(\Omega)$ for all $k$, such that $u_{R_j}\to u$ pointwise on $\Omega$ in $C^k(K)$ for any fixed compact $K\Subset\Omega$, and $u$ solves \eqref{EL}.  We must show that $u\in X$.

The next step is to show that 
\be\label{msame}
m_0=m=\lim_{R\to\infty} m_{0,R}.
\ee
Assuming \eqref{msame} true for the moment, we show that the $u$ obtained above (as limits of the minimizers $u_{R_j}$ in bounded regions) is indeed a minimizer of $E_\eps$ in $X$.
For any fixed $R_1$, strong convergence on compact sets implies that
$$  \int_{B_{R_1}\setminus \Omega_0} e_\eps(u)\, dx
   = \lim_{R\to\infty} \int_{B_{R_1}\setminus \Omega_0} e_\eps(u_R)\, dx
    \le \lim_{R\to\infty} m_{0,R} = m.  $$
Taking the supremum over $R_1$, we conclude that $E_\eps(u)\le m$.
Since the energy is finite, we may then apply the estimates of \cite{BMR} to conclude that $|u|\to 1$ as $|x|\to\infty$, and $\deg(\frac{u}{|u|},\infty)=0$.  Finally, by \cite{Shafrir},  there exists $\phi_0\in\R$ with $u(x)\to e^{i\phi_0}$ as $|x|\to\infty$.  Thus, $u\in X$, and attains the minimum of $E_\eps$.  

In the case that $\Omega_0=B_{R_0}$, suppose $u$ attains the minimum in $X$, and $u(x)\to e^{i\phi_0}$ as $|x|\to\infty$ with $\phi_0\in(-\pi,\pi]$ and $\phi_0\neq 0$.  Using complex notation $z=x+iy$ for $z\in\C\setminus B_{R_0}\simeq\R^2\setminus B_{R_0}$, define $v(z)=e^{-i\phi_0}u(ze^{i\phi_0/D})$.  Then, $v\in X_0$, and since $e^{-i\phi_0}g(ze^{i\phi_0/D})=g(z)$ for $g(z)=e^{iD\theta}$, we have $E_\eps(v)=E_\eps(u)$.  Since $m_0=m$, $v$ attains the minimum of $E_\eps$ in $X_0$ as desired.

To conclude the proof, it remains to verify the claim \eqref{msame}
On one hand, if we define $\tilde u_R$ as the extension of $u_R$ to $\Omega$ with $\tilde u_R(x)=1$ for $x\in\R^2\setminus B_R$, then 
$\tilde u_R\in X_0$ and
$E_\eps(\tilde u_R)=E_\eps(u_R; B_R\setminus\Omega_0)=m_{0,R}$.  In particular, we conclude that 
$$m\le m_0\le m_{0,R}$$
 holds for all $R$.
To obtain a complementary bound, let $\eta>0$ be given, and choose $u\in X$ with $E_\eps(u)\le m+\frac1{10}\eta$.
Since $u\in X$, there exists $\phi_0\in (-\pi,\pi]$ with $u(x)\to e^{i\phi_0}$ as $|x|\to\infty$.  Since $|u(x)|\to 1$, we may choose $R$ sufficiently large that $u(x)=\rho(x)e^{i a(x)}$ for $|x|\ge R$, with $\rho(x)=|u(x)|>\frac12$ and $|a(x)-\phi_0|<\frac{\eta}{ 10}$ for $|x|\ge R$.  By making $R$ larger if necessary, we may also assume
\be\label{smallen}
E_\eps(u;\R^2\setminus B_R) < \frac{\eta}{ 10}.
\ee
Define a family of cut-off functions, 
$$  \chi_{N,R}(x)=\begin{cases}
  0, &\text{if $r\le R$,}\\
  \frac{\ln (r/R)}{\ln N}, &\text{if $R<r<NR$},\\
  1, &\text{if $r\ge NR$.}
\end{cases}
$$

Now define $\tilde u(x):= \tilde\rho(x)e^{i\tilde a(x)}$, where
$$\tilde\rho(x):= \chi_{N,R}(x) + (1-\chi_{N,R}(x))\rho(x), \qquad
   \tilde a(x):= (1-\chi_{N,R}(x)) a(x).
$$
Then, $\tilde u\in X_{0,NR}$, and using \eqref{smallen}, $|a(x)|\le |\phi_0|+{\eta\over 10}<2\pi$, and $\frac12<\rho(x)\le\tilde\rho(x)\le 1$ for $|x|\ge R$, we have
\begin{align*}
E_\eps(\tilde u) &\le E_\eps(u; B_R) +
   \frac12\int_{R\le |x|\le NR} \left( |\nabla\tilde\rho|^2 + 
     \tilde\rho^2 |\nabla \tilde a|^2 + \frac{1}{ 2\eps^2} (1-\tilde\rho^2)^2\right) dx \\
&\le E_\eps(u; B_{NR}) + \frac12\int_{R\le |x|\le NR} \left( |\nabla\tilde\rho|^2 + 
     \tilde\rho^2 |\nabla \tilde a|^2
        \right) dx \\
 &\le m+\frac{\eta}{ 10} + \int_{R\le |x|\le NR} \left(
    |\nabla\rho|^2 + (1-\rho)^2|\nabla\chi_{N,R}|^2 + |\nabla a|^2 + a^2|\nabla\chi_{N,R}|^2 \right) dx \\
    &\le  m+\frac{\eta}{ 10} + 8E(u; \R^2\setminus B_R) 
    + 8\pi^3\int_{R}^{NR} [\ln N]^{-2} \frac{dr}{ r} \\
  &\le m +\frac{9\eta}{ 10} + \frac{8\pi^3}{\ln N}.
\end{align*}
Choosing $N_0$ sufficiently large that $\frac{8\pi^3}{\ln N_0}<\frac{\eta}{ 10}$, we obtain functions $\tilde u\in X_{0,NR}$, for all $N\ge N_0$, with $m_{0,NR}\le E(\tilde u)\le m +\eta$.  Thus, we have
$$ \limsup_{R\to\infty} m_{0,R} \le m\le m_0 \le \inf_R m_{0,R},  $$
and the claim \eqref{msame} is established.
\end{proof}

\section{Some Basic Estimates}
In this section we prove two fundamental estimates:  a rough upper bound on the energy of minimizers, and a pair of {\it a priori} pointwise bounds for all solutions of the Euler-Lagrange equations \eqref{EL}.

\begin{lemma}\label{upperbound}
Let 
$$\DDD=\deg(g;\Gamma)>0.
$$
For each problem $i=$I, II, III, there exists a constant $C=C(g,\Gamma)$, independent of $\eps$, for which
\be\label{ub}
  \inf_{u\in \mathbb{H}_i} E_\eps(u) \le \pi\min\{2\alpha,1\}\DDD \, |\ln\eps| + C.
  \ee
\end{lemma}

\begin{proof}[]
For $\alpha>\frac12$, we choose a test function $u_\eps$ as in \cite{BBH2}.  This is a standard procedure, so we merely describe the steps to take in each problem, I, II, III.
In problem I, $\Gamma=\partial\Omega$, so this is done exactly as in \cite{BBH2}, treating the weak anchoring condition as a Dirichlet condition, and defining an $S^1$-valued map $v_\eps$ in the complement of $\mathcal{D}$ disks of radius $\eps$, with degree one on the boundary of each excised disk and $v_\eps=g$ on $\partial\Omega=\Gamma$.  For problem II, we again treat the weak anchoring condition as a Dirichlet condition, but the function $v_\eps$ is chosen with degree $-1$ on each excised disk.  For problem III, it suffices to take $v_\eps$ constructed for problem II in $\Omega=B_R\setminus\Omega_1$, and extend $v_\eps=1$ in $\R^2\setminus B_R$.  For each problem, we obtain the same upper bound,  $E_\eps(u_\eps)\le \pi \DDD\, |\ln\eps| + C$, when $\alpha>\frac12$.

For $0<\alpha\le\frac12$, we construct functions $u_\eps$ with constraint $|u_\eps|=1$, using the technique of Kurzke \cite{Ku}.  As our weak coupling condition is subtly different from his, we give some details of the construction below.

We choose $\mathcal{D}$ points $q_1,\dots,q_{\mathcal{D}}\in\Gamma$ which are well separated, and let $R<\frac12 |q_i-q_j|$, for all $i\neq j$.  For each $q_i$, we first define $v_\eps=v_\eps^{(i)}$ in $\omega_R(q_i)=B_R(q_i)\cap\Omega$.  Let $\tau_i$ be the tangent vector to $\Gamma$ at $q_i$, oriented in the same direction as $\Gamma$.  We introduce polar coordinates $(r,\theta)$ centered at $q_i$, with angle $\theta$ measured from the ray defined by the oriented tangent vector $\tau$.  Since $\Gamma$ is smooth, by choosing $R$ sufficiently small we may ensure that the domain $\omega_R(q_i)$ is a polar rectangle:  there exist $C^1$ functions $\theta_1(r),\theta_2(r)$, so that
$$   \omega_R(q_i)=\{(r,\theta): \ \theta_1(r)<\theta<\theta_2(r), \ 0<r<R\}. 
$$
Furthermore, there exists a constant $c_1$ for which $|\theta_1(r)|\le cr$ and $|\pi-\theta_2(r)|\le cr$.  

Let $\gamma$ be a lifting of $g$ on the arc $\Gamma\cap B_R(q_i)$, so $g=e^{i\gamma}$ on this arc.  Our choice of coordinates in $\omega_R(q_i)$ divides $\Gamma\cap B_R(q_i)\setminus\{q_i\}$ into two pieces, $\Gamma_1,\Gamma_2$, parametrized by $(r,\theta_1(r)), (r,\theta_2(r))$, $0<r<R$, respectively. (See Figure~\ref{UBfig}.)
\begin{figure}[!htbp]
\begin{center}
\includegraphics*[width=200pt]{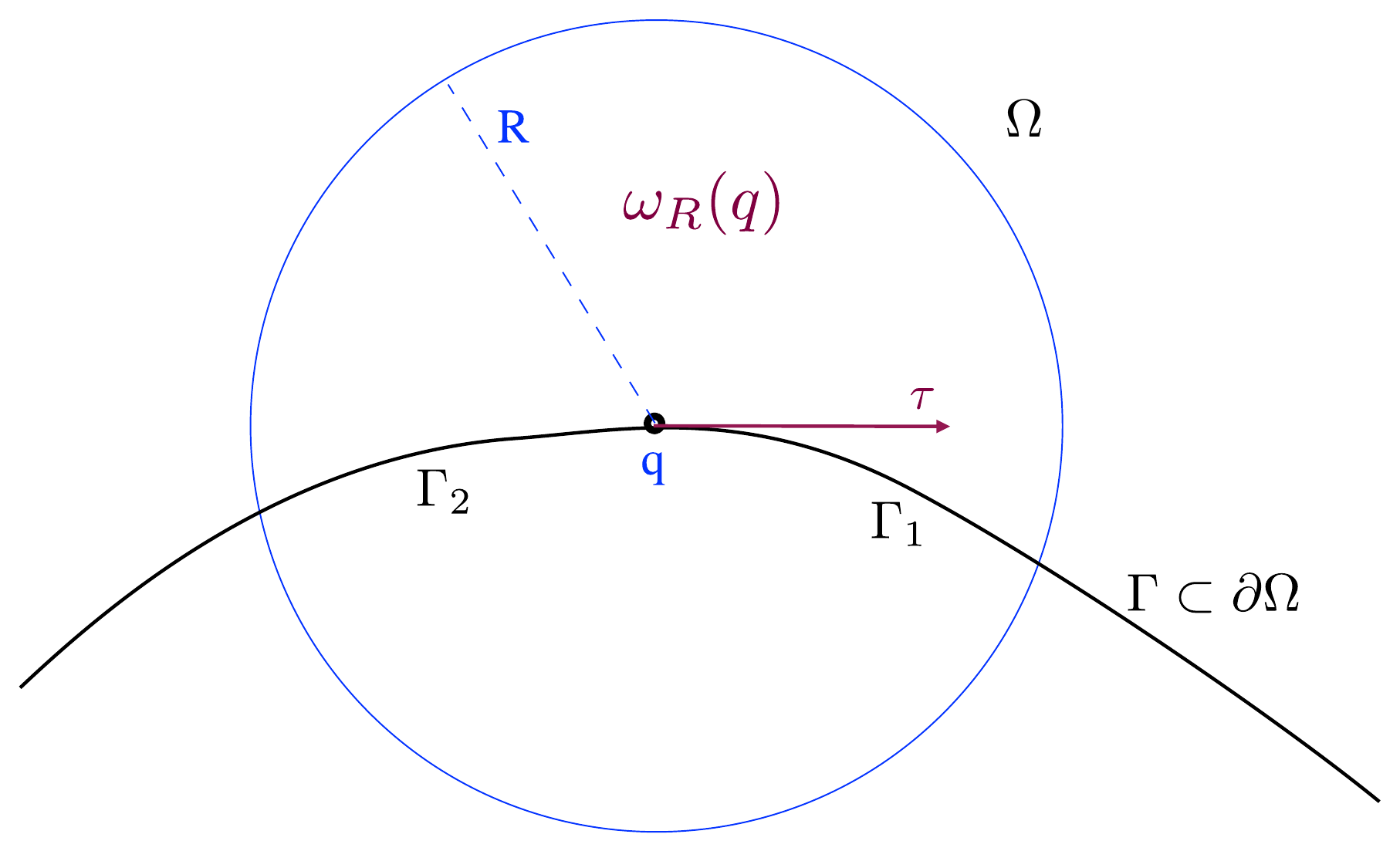}
\end{center}
\caption{The domain $\omega_R(q)=B_R(q)\cap\Omega$, used in the upper bound construction.}\label{UBfig}
\end{figure}

Define
$$   h_1(r)= \gamma\left(e^{i\theta_1(r)}\right), \qquad
      h_2(r)= \gamma\left(e^{i\theta_2(r)}\right)+2\pi.
$$
Following \cite{Ku}, we now define an $S^1$-valued function in $\omega_R(q_i)\setminus\{q_i\}$ via its phase,
$$   \phi(r,\theta) = \frac{h_2(r)-h_1(r)}{\theta_2(r)-\theta_1(r)}
                   \left(\theta - \theta_1(r)\right) + h_1(r).
$$
Note that on $\Gamma_j$, $j=1,2$, we have $\phi(r,\theta_j(r))=h_j(r)$, and so $e^{i\phi}=g$ on $\Gamma\setminus\{q_i\}$.  Finally, we define a cutoff near $q_i$, $\chi_\eps(r)\in C^\infty$, with $0\le\chi_\eps(r)\le 1$ for all $r$, $\chi_\eps(r)=0$ for $r<\eps^\alpha$, and $\chi_\eps(r)=1$ for $r\ge 2\eps^\alpha$.  The desired test configuration in $\omega_R(q_i)$ is then
$$   v_\eps=v_\eps^{(i)} = \text{exp}\,\left\{
       i\left[ \chi_\eps(r)\phi(r,\theta) + (1-\chi_\eps(r))\gamma(q_i)\right]\right\}.
$$
We observe that the phase of $v_\eps$ turns by approximately $2\pi$ on the approximate semicircle $\partial\omega_R(q_i)$, as opposed to the construction in \cite{Ku} in which the phase rotates by only $\pi$.

Since $|v_\eps|=1$ in $\omega_R(q_i)$ and $v_\eps=g$ on $\Gamma\setminus B_{2\eps^\alpha}(q_i)$, $i=1,2$, we have
$$  \frac{1}{\eps^2} \int_{\omega_R(q_i)}(|v_\eps|^2-1)^2 \, dx =0, \quad
\lambda\int_{\Gamma\cap B_R(q_i)} |v_\eps-g|^2\, ds \le c_2, $$
with constant $c_2$ independent of $\eps.$
A straightforward calculation also shows that both
$$ \int_{\omega_R(q_i)} |\partial_r v_\eps|^2\, dx, \ 
 \int_{\omega_{2\eps^\alpha}(q_i)} |\partial_\theta v_\eps|^2\, dx \le c_3,
$$
are uniformly bounded in $\eps$.  So the main contribution comes from the theta derivative in the annular region, $A_{R,\eps^\alpha}=\omega_R(q_i)\setminus \omega_{\eps^\alpha}(q_i),$
\begin{align*}
\int_{A_{R,\eps^\alpha}} \frac12 |\partial_\theta v_\eps|^2\, dx
& = \frac12 \int_{\eps^\alpha}^R 
     \frac{(h_2(r)-h_1(r))^2}{\theta_2(r)-\theta_1(r)} \frac{dr}{r} \\
&\le \frac12 \int_{\eps^\alpha}^R 
       \frac{(2\pi+ c_1 r)^2}{(2\pi - c_1 r)} \frac{dr}{r} \\
&\le 2\pi\alpha\ln \left(\frac{1}{\eps}\right) + c_4.
\end{align*}

Next we construct $v_\eps$ in 
$\tilde\Omega= \Omega\setminus\bigcup_{i=1}^{\mathcal{D}}\omega_R(q_i)$.
Let $\tilde\Gamma$ denote the closed contour which follows $\Gamma$ away from $\omega_R(q_i)$, $i=1,\dots,\mathcal{D}$, and $\partial\omega_R(q_i)\cap\Omega$.  
We then define $\tilde g:\ \tilde\Gamma\to S^1$ by $\tilde g=g$ on 
$\Gamma\setminus\bigcup_{i=1}^{\mathcal{D}}\omega_R(q_i)$ and 
$\tilde g= v_\eps^{(i)}$ on $\partial\omega_R(q_i)\cap\Omega$.
Orienting $\tilde\Gamma$ in the same sense as $\Gamma$ where they coincide, we note that the arcs along $\partial\omega_R(q_i)\cap\Omega$ are negatively oriented, and so the phase of $\tilde g$ turns by $-2\pi$ along each of these circular arcs.  In particular, $\deg(\tilde g; \tilde\Gamma)=0$.
Thus, we may define $v_\eps$ in $\tilde\Omega$ as the $S^1$-valued harmonic extension of $\tilde g$ to $\tilde\Omega$, which has bounded energy,
$$  \int_{\tilde\Omega}\frac12 |\nabla v_\eps|^2\, dx \le c_5.  $$
Putting these pieces together, when $0<\alpha\le\frac12$, we obtain $v_\eps$, with $|v_\eps|=1$ in all $\Omega$, and with the estimate 
$$  E_\eps(v_\eps)\le
       2\alpha\pi \DDD \ln\frac1\eps + C,
$$
as desired.
\end{proof}

We have the following pointwise upper bounds on solutions to \eqref{EL}.

\begin{lemma}\label{apriori}
Let $u_\eps$ be any solution of \eqref{EL}.  Then $|u_\eps(x)|\le 1$ and there exists a constant $C_0=C_0(\Omega)>0$ so that $|\nabla u_\eps|\le C_0/\eps$, for all $x\in\Omega$.
\end{lemma}

\begin{proof}[]
Let $u$ solve \eqref{EL}, in settings I, II, or III, and set $V=|u|^2-1$.  Then, 
$\nabla V=2 u\cdot\nabla u$ and $\frac12\Delta V \ge \frac{1}{\eps^2}(V+1)V$ in $\Omega$.  In problems I, II, we multiply this inequality by $V_+=\max\{V,0\}$, and integrate over $\Omega$, to obtain:
\be\label{maxpr}
  0\le \frac{1}{\eps^2}\int_\Omega |u|^2 V_+ \le
\frac12\int_{\partial\Omega} V_+ \frac{\partial V}{\partial\nu} ds
-\frac12\int_\Omega |\nabla V_+|^2.  
\ee
On $\Gamma\subset\partial\Omega$, we have
$$  
  V_+\frac{\partial V}{\partial\nu} = -2V_+\lambda u\cdot(u-g)\le 0, $$
since $|u|^2 - u\cdot g\ge |u|(|u|-1)\ge 0$ when $V_+\neq 0$.  On $\partial\Omega\setminus\Gamma$, $|u|=1$ so $V_+=0$, and hence the boundary integral in \eqref{maxpr} is nonpositive.  Hence, \eqref{maxpr} implies 
\be\label{maxpr2}
0\le \frac{1}{\eps^2}\int_\Omega |u|^2 V_+ \le-\frac12\int_\Omega |\nabla V_+|^2\le 0,
\ee
 and hence both integrals are zero.  In conclusion, $V_+\equiv 0$, and $|u|\le 1$ in $\Omega$.

For the exterior problem III, by the definition of the spaces $X, X_0$ and the finiteness of the energy $E_\eps(u)$, there exists a sequence $R_n\to\infty$ such that  $|u(R_n,\theta)|\le 2$ and 
$$   \int_0^{2\pi} \left[\frac12 |\nabla u(R_n, \theta)|^2 
               + \frac{1}{4\eps^2} (|u(R_n,\theta)|^2-1)^2\right] R_n\, d\theta
               \to 0.  $$
 As above, we multiply the inequality for $V$ by $V_+$, but now integrate over $\Omega\cap B_{R_n}$ to obtain an inequality as in \eqref{maxpr}.
 The boundary term on the right hand side may be estimated as:
\begin{multline*}  \left|\int_{\partial B_{R_n}}  V_+ \frac{\partial V}{\partial \nu} ds\right|
    = 2\left| \int_0^{2\pi} (|u(R_n,\theta)|^2-1)_+ u(R_n,\theta)\cdot \frac{\partial u}{\partial r}(R_n,\theta)\, R_n\, d\theta \right|  \\
 \le 
 4\int_0^{2\pi} \left[ |\nabla u(R_n, \theta)|^2 
               +  (|u(R_n,\theta)|^2-1)^2\right] R_n\, d\theta
               \to 0.
\end{multline*} 
Passing to the limit $R_n\to\infty$, we arrive at the same string \eqref{maxpr2} of inequalities, and hence $|u|\le 1$ as before.

To establish the gradient bound, we argue by contradiction:  suppose there exist sequences $\eps_k\to 0$, $x_k\in \overline{\Omega}$ for which
$t_k:=|\nabla u_k(x_k)| = \|\nabla u_k\|_\infty$ satisfies $t_k\eps_k\to\infty$.  Blowing up at scale $t_k$ around the points $x_k$, define
$v_k(x):= u_k\left(x_k + \frac{x}{t_k}\right)$.  By our choice of scaling, $\|v_k\|_\infty =1$, and $v_k$ solves 
$$  -\Delta v_k = \frac{1}{(t_k\eps_k)^2} (|v_k|^2-1)v_k \to 0, $$
uniformly on $\Omega$ (since $\|u_k\|_\infty=\|v_k\|_\infty\le 1$, by the first part of the lemma.)  If, for some subsequence, $t_k\text{dist}\,(x_k,\partial\Omega)\to \infty$, then the domain $t_k[\Omega-x_k]$ of $v_k$ converges to all $\R^2$, and $v_k\to v$ in $C^k_{loc}$.  Moreover, the limit
 $v$ is a bounded harmonic function on $\R^2$, and hence constant:  $\nabla v(x)\equiv 0$.  However, by construction, $|\nabla v_k(0)|=1$ for all $k$, and hence $|\nabla v(0)|=1$, a contradiction.
 
On the other hand, if $t_k\text{dist}\,(x_k, \partial\Omega)$ is uniformly bounded, then the domains $t_k[\Omega-x_k]$ of $v_k$ converge to a half-space $\R^2_+$, with boundary condition 
$$  \frac{\partial v_k}{\partial\nu}= -\frac{\lambda}{t_k}\left[v_k-g\left(x_k + \frac{x}{t_k}\right)\right] \to 0.  $$
That is, $v_k\to v$ which is  bounded and harmonic in $\R^2_+$, and with a Neumann condition $\partial_\nu v=0$ on the boundary.  By the reflection principle and Liouville's theorem we again conclude that $v$ is constant, which leads to the same contradiction as in the previous case.  Thus, the desired gradient bound must hold. 
 \end{proof}

\section{$\eta$-compactness}

We begin by proving an $\eta$-compactness (or $\eta$-ellipticity) result (see \cite{Struwe}, \cite{Riviere}).  Basically, if the energy contained in a ball of radius $\eps^\beta$ is too small, there can be no vortex in a slightly smaller ball, $B_{\eps^\gamma}(x_0)$.  To this end, we recall that $\lambda=\lambda(\eps)=K\eps^{-\alpha}$ for $\alpha\in (0,1]$, $K>0$ constant, and fix $\beta,\gamma$ such that $\frac34\alpha\le\beta<\gamma<\alpha$.

\begin{proposition}[$\eta$-compactness]\label{eta}
There exist constants $\eta, C, \eps_0>0$ such that for any solution $u_\eps$ of \eqref{EL} with $\eps\in (0,\eps_0)$, if $x_0\in\overline{\Omega}$ and
\be\label{eta0}
E_\eps\left(u_\eps;B_{\eps^\beta}(x_0)\right) 
   \le \eta\, |\ln\eps|,
\ee
then
\begin{gather}
\label{eta1}  |u_\eps|\ge \frac12 \quad\text{ in } \quad B_{\eps^\gamma}(x_0),  \\
\label{eta2}  |u_\eps - g| \le \frac14 \quad\text{ on $\Gamma\cap B_{\eps^\gamma}(x_0)$,} \\
\label{eta3}  
   \frac{1}{4\eps^2}\int_{B_{\eps^\gamma}(x_0)} \left(|u_\eps|^2-1\right)^2\, dx + 
      \frac{\lambda}{2}\int_{\Gamma\cap B_{\eps^\gamma}(x_0)} |u_\eps-g|^2 \, ds \le C\eta.
\end{gather}
\end{proposition}

We note that in case $\Gamma\cap B_{\eps^\beta}(x_0)=\emptyset$, this has been proven in Lemma~2.3 of \cite{Struwe}, and hence it suffices to consider $x_0\in \Gamma\subset\partial\Omega$ when proving Proposition~\ref{eta}.

Define $\Gamma_r(x_0)= \partial\Omega\cap B_r(x_0)$, and following Struwe \cite{Struwe},
\be\label{Fdef}
F(r)=F(r; x_0, u,\eps)=r\left[
  \int_{\partial B_r(x_0)\cap\Omega} \left\{ |\nabla u|^2
      + \frac{1}{2\eps^2}(|u|^2-1)^2\right\} ds 
       + \lambda(\eps)\sum_{x\in\partial \Gamma_r(x_0)} |u(x)-g(x)|^2
  \right].
\ee
Note that if $\partial\Gamma_r(x_0)\neq\emptyset$, then for $r>0$ sufficiently small it consists of two points.

The proof of Proposition~\ref{eta} relies on the following estimate.  For any $x_0\in\overline{\Omega}$ and $R>0$, we define (as in the proof of Lemma~\ref{upperbound}) 
$$ \omega_R(x_0)=B_R(x_0)\cap\Omega . 
$$
 Then, we prove:
\begin{lemma}\label{poho}
There exist $C>0$ and $r_0>0$ such that for $\eps \in(0,1)$, $x_0\in\Gamma$, and $r \in (0,r_0)$, we have that
\begin{equation*}
	\frac{1}{2\eps^2} \int_{\omega_r(x_0)} \big( |u_\eps|^2-1\big)^2 \, dx + \lambda \int_{\Gamma_r(x_0)} |u-g|^2 \,dS 
	\leq C  \left[ 
		r\int_{\omega_r(x_0)} |\grad u_\eps|^2 \, dx 
		+ F(r) + r^2\lambda
	\right].
\end{equation*}
\end{lemma}

\begin{proof}[ of Lemma~\ref{poho}]
We denote $u=u_\eps$, $\omega_r=\omega_r(x_0)$, and $\Gamma_r=\Gamma_r(x_0)$ for convenience, as $x_0\in\Gamma$ and $\eps>0$ are fixed.

Let $\psi \in C^\infty(\Omega;\R^2)$ be a vector field, to be determined later.  Taking the complex scalar product of the equation \eqref{EL} with $\psi\cdot\nabla u$ and integrating over $\omega_r$, we obtain the Pohozaev-type equality,
\begin{multline}\label{pohoz}
\int_{\partial\omega_r} \left\{
-(\partial_\nu u, \psi\cdot\nabla u) + \frac12 |\nabla u|^2 (\psi\cdot \nu)
  +\frac{1}{4\eps^2} (|u|^2-1)^2(\psi\cdot \nu)
\right\} ds  \\
= \int_{\omega_r} \left\{ \frac{1}{4\eps^2} (|u|^2-1)^2 \div\psi 
             + \frac12 |\nabla u|^2\div \psi - \sum_{i,j}\partial_i\psi_j (\partial_iu,\partial_ju)
             \right\}dx.
\end{multline}

We choose $r_0>0$ sufficiently small so that $\Gamma\cap B_r(x_0)$ consists of a single smooth arc, and $\omega_r$ is strictly starshaped with respect to some $x_1\in\omega_r$, for all $0<r\le r_0$.  

Let $\mathcal{N}$ be a $2r_0$-neighborhood of $\Gamma$.  We claim that, by taking $r_0$ smaller if necessary, there exists a vector field $X\in C^2(\mathcal{N};\R^2)$ with the following properties (see \cite{Ku}, \cite{Moser}):
\begin{gather}
X\cdot\nu = 0, \quad\text{for all $x\in \Gamma_r$}, \label{X1} \\
|X-(x-x_0)|\le C |x-x_0|^2, \quad |DX - Id| \le C |x-x_0|,\quad
\text{for all $x\in \omega_r$}, \label {X2}
\end{gather}
for a constant $C>0$, for any $x_0\in\Gamma$.
The existence of such a vector field in a disk $B_r(x_0)$ follows from the smoothness of $\Gamma$; to obtain the uniform global estimates \eqref{X1}, \eqref{X2} we use the compactness of $\Gamma$ and a partition of unity.
In particular, note that $X= (X\cdot\tau)\tau\simeq (x-x_0)\tau$ lies along the tangent vector on $\Gamma_r$.  

We now take $\psi=X$ in \eqref{pohoz} and estimate each term in \eqref{pohoz}, separating the $\partial\omega_r$ terms into the pieces along $\Gamma_r$ and along $\partial B_r(x_0)\cap\Omega$.  First,
on $\Gamma_r$ we have $X\cdot\nu=0$, and the only contribution to the left hand side of \eqref{pohoz} is:
\begin{align} \nnn
-\int_{\Gamma_r} (\partial_\nu u, \psi\cdot\nabla u)\, ds
  &= \lambda \int_{\Gamma_r} (u-g\, , \, (X\cdot\tau) \partial_\tau u)\, ds \\
  &= \lambda \int_{\Gamma_r} \left[ 
     \left( u-g \, ,\, \partial_\tau(u-g)\right) + ( u-g\, , \, \partial_\tau g)\right]
       X\cdot\tau\, ds
\label{p1}
\end{align}
The first term in \eqref{p1} may be evaluated by integration by parts:
\begin{align}\nnn
  \lambda \int_{\Gamma_r} 
     \left( u-g \, ,\, \partial_\tau(u-g)\right) ds
     &=\frac{\lambda}{2} \int_{\Gamma_r} \partial_\tau
     \left( |u-g|^2\right) (X\cdot\tau) ds \\
\nnn
  & = \frac{\lambda}{2}\left[ |u-g|^2 (X\cdot\tau) |_{\partial\Gamma_r}
   - \int_{\Gamma_r} |u-g|^2 \partial_\tau(X\cdot\tau)\, ds\right].
\nnn
\end{align}
On the endpoints of $\Gamma_r$, $|X\cdot\tau \mp r|\le Cr^2$ and on $\Gamma_r$ itself, $\partial_\tau(X\cdot\tau)= 1 + O(|x-x_0|)$, by \eqref{X2}.  Hence, there exists a constant $C>0$ for which
\be\label{p2}
\lambda \int_{\Gamma_r} 
     \left( u-g \, ,\, \partial_\tau(u-g)\right) ds \le
       \frac{\lambda}{2}\left[ -\int_{\Gamma_r} |u-g|^2\, ds
           + r\sum_{\partial\Gamma_r} |u-g|^2\right]  + C\lambda r^2.
\ee
For the second term of \eqref{p1}, we have the rough estimate
\be\label{p3}
\left|\lambda\int_{\Gamma_r} (u-g\, , \, \partial_\tau g)(X\cdot\tau)\, ds\right|
   \le C\|g\|_{C^1}\lambda r^2.
\ee
The remaining terms on the left-hand side of \eqref{pohoz} may also be estimated in a simple way, using $|X\cdot\nu|, |X\cdot\tau|\le Cr$:
\begin{gather}\label{p4}
\left|\int_{\partial\omega_r\cap\Omega}\left[ (u-g\, , \, \partial_\tau g)(X\cdot\tau) -\frac12 |\nabla u|^2 (X\cdot\nu)\right]\, ds\right|
   \le Cr\,\int_{\partial\omega_r\cap\Omega}|\nabla u|^2\, ds, \\
\label{p5}
\frac{1}{4\eps^2}\int_{\omega_r} (|u|^2-1)^2\left(X\cdot\nu\right) \, ds
 = \frac{1}{4\eps^2}\int_{\partial B_r\cap\Omega} 
       (|u|^2-1)^2\left(X\cdot\nu\right) \, ds
       \le \frac{Cr}{\eps^2}\int_{\partial B_r\cap\Omega} 
       (|u|^2-1)^2\left(X\cdot\nu\right) \, ds.
\end{gather}
For the terms on the right side of \eqref{pohoz}, we use \eqref{X2}:  $|\partial_i X_j - \delta_{ij}|\le Cr$, and for $r_0$ chosen smaller if necessary, we may assume $\div X\ge 2-Cr>1$ in $\omega_r$.  Thus, the right side of \eqref{pohoz} may be estimated as:
\begin{multline}\label{p6}
\int_{\omega_r} \left\{ \frac{1}{4\eps^2} (|u|^2-1)^2 \div X 
             + \frac12 |\nabla u|^2\div X- \sum_{i,j}\partial_i X_j (\partial_iu,\partial_ju)
             \right\}dx \\
\ge 
  \int_{\omega_r} \left\{ \frac{1}{4\eps^2} (|u|^2-1)^2
     - C r |\nabla u|^2 \right\} dx.
\end{multline}
Putting the above estimates together, we arrive at the desired bound.
\end{proof}

\begin{proof}[ of Proposition~\ref{eta}]
We follow \cite{Struwe}, \cite{Moser}.  If $x_0\in\Omega\setminus\Gamma$, this is proven in \cite{Struwe}, so we restrict our attention to $x_0\in\Gamma$.

Since
\be\label{Fest}  \eta\ln\frac{1}{\eps}\ge E_\eps(u_\eps; \omega_{\eps^\beta}\setminus\omega_{\eps^\gamma}) = \int_{\eps^\gamma}^{\eps^\beta}
 \frac{F(r)}{r} \, dr,
\ee
there exists $r_\eps\in (\eps^\gamma,\eps^\beta)$ so that
$$  F(r_\eps) \le \frac{\eta}{\gamma-\beta}.  $$
By Lemma~\ref{poho} and the upper bound \eqref{ub}, we deduce \eqref{eta3}.  

Suppose that for some $x_2\in B_{\eps^\gamma}(x_0)$ it were true that $|u_\eps(x_2)|<\frac12$.  
By Lemma~\ref{apriori}, $|\nabla u_\eps|\le C_0/\eps$, so it would follow that $|u_\eps(x)|<\frac34$ for $x\in B_{\eps/4C_0}(x_2)$.  But then,
$$  \frac{1}{4\eps^2}\int_{B_{\eps^\gamma}(x_0)} (|u_\eps|^2-1)^2
\ge \frac{1}{4\eps^2}\int_{B_{\eps/4C_0}(x_2)}(|u_\eps|^2-1)^2
\ge  \frac{49\pi}{2^{14}C_0}, $$
which contradicts \eqref{eta3} provided $\eta$ is chosen small enough.  Thus, for the appropriate choice of $\eta$ (which is independent of $x_0$), we must have \eqref{eta1} verified.

To verify \eqref{eta2}, we return to the Pohozaev identity \eqref{pohoz}.
We recall that for $r=r_\eps$ (as in the proof of \eqref{eta3}) sufficiently small, the smoothness and compactness of $\Gamma$ ensure that $\omega_r$ is strictly starshaped around some $x_1\in\omega_r$, and for $\eps_0$ chosen sufficiently small, we have $(x-x_1)\cdot\nu\ge r/4$ on $\partial\omega_r$.
We apply \eqref{pohoz} with vector field $\psi= x-x_1$, and obtain:
\be\label{p7}
\int_{\partial\omega_r} \left\{
(x-x_1)\cdot\nu\left[ |\partial_\tau u_\eps|^2 - |\partial_\nu u_\eps|^2\right] + (x-x_1)\cdot\tau\, (\partial_\nu u_\eps , \partial_\tau u_\eps)\right\} ds\le
\frac{1}{\eps^2}\int_{\omega_r} (1-|u_\eps|^2)^2 dx.
\ee
Using Cauchy-Schwartz,
$$  \left| \int_{\partial\omega_r}  (x-x_1)\cdot\tau\, (\partial_\nu u_\eps , \partial_\tau u_\eps)
\right|
\le  \int_{\partial\omega_r} \left\{ \frac{r}{8} |\partial_\tau u_\eps|^2
          + 2r |\partial_\nu u_\eps|^2 \right\} ds,
$$
and hence
\begin{align*}  \int_{\partial\omega_r} |\partial_\tau u_\eps|^2 \, ds
&\le C\int_{\partial\omega_r} |\partial_\nu u_\eps|^2\, ds + 
\frac{1}{r\eps^2}\int_{\omega_r} (1-|u_\eps|^2)^2 \\
  &= C\lambda^2\int_{\Gamma_r} |u_\eps-g|^2\, ds 
     + C\eps^{-\gamma}\\
     &\le C\eps^{-\alpha},
\end{align*}
using Lemma~\ref{poho} and \eqref{eta3}.  By the Sobolev embedding theorem (on the one-dimensional set $\Gamma_r$,) there exists a constant $C>0$ (again, independent of $x_0$) for which
$$  |u_\eps(x)-u_\eps(y)| \le C\sqrt{|x-y|}\eps^{-\alpha/2} $$
holds for all $x,y\in\Gamma_r$.

The conclusion now follows as in Proposition~3.6 of \cite{Ku}.  Assume there exists $x_2\in\Gamma_r$ for which $|u_\eps(x_2)-g(x_2)|>\frac14$.  By the same argument as in the proof of \eqref{eta1}, there would exist a radius $\rho=c\eps^{\alpha}$, for constant $c>0$ independent of $x_0$, for which
$|u_\eps(x)-g(x)|>\frac18$ when $x\in \Gamma_r\cap B_{c\eps^\alpha}(x_2)$.  In that case, we would have
$$  C\eta\ge \lambda\int_{\Gamma_r\cap B_{c\eps^\alpha}} |u-g|^2\, ds
> \frac{Kc^2}{64}, $$
which would lead to a contridiction for $\eta$ chosen sufficiently small.  By reducing the value of $\eta$ required for the proof of \eqref{eta1} if necessary, we obtain \eqref{eta2}.  This completes the proof of Proposition\ref{eta}.
\end{proof}

Next we estimate the energy contribution near a vortex.  For $x_0\in \overline{\Omega}$, denote by
$$  A_{r,R}(x_0)=w_R(x_0)\setminus w_r(x_0) .   $$
In case $x_0\in\Gamma$, for $R$ sufficiently small the piece of the boundary $\partial A_{r,R}(x_0)\cap\partial\Omega$ consists of exactly two arcs along $\Gamma_R=\Gamma\cap B_R(x_0)$, which we will denote by $\Gamma_{r,R}^\pm$.  (See Figure~\ref{fig:Ar}.)  

\begin{figure}[!htbp]
\begin{center}
\includegraphics*[width=200pt]{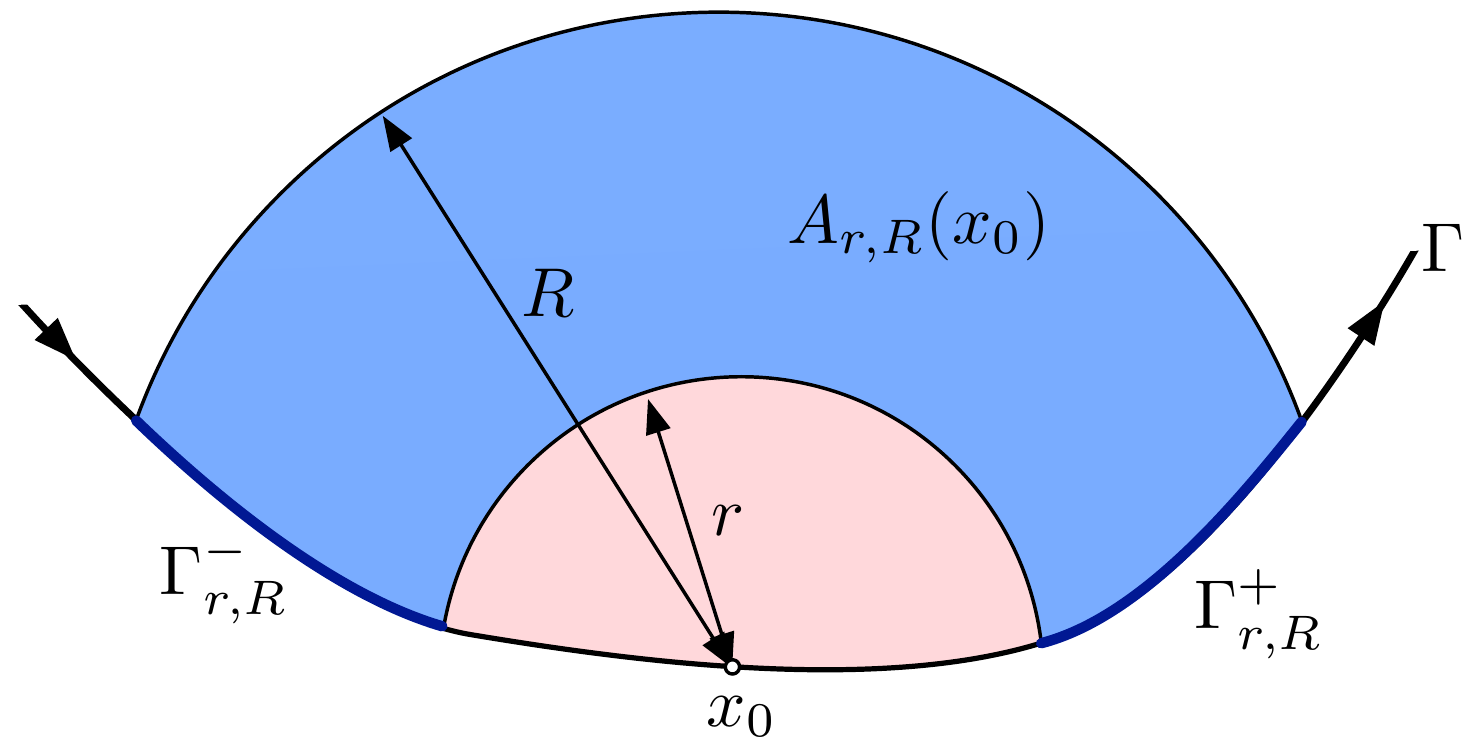}
\end{center}
\caption{Annulus $A_{r,R}$.}\label{fig:Ar}
\end{figure}

We now define a degree for nonvanishing maps $u$ on $A_{r,R}(x_0)$.
Assume that $|u|\ge \frac12$ on $A_{r,R}$ and $|u-g|\le\frac14$ on $\Gamma_{r,R}^\pm$.  If $A_{r,R}\cap\Gamma=\emptyset$, we may define the degree $\deg(\frac{u}{|u|};\partial A_{r,R}(x_0))=d$ in the usual way.  For $x_0\in\Gamma$, we define it as follows.  Since $|u-g|\le\frac14$ on $\Gamma_{r,R}^\pm$ and $g$ is smooth, we may extend $u$ to $\tilde u$ on all of $\Gamma_R$ in such a way that $u$ is smooth and satisfies $|\tilde u-g|\le\frac12$ on all of $\Gamma_R$.  Setting $\tilde u=u$ on $\partial B_R(x_0)\cap\Omega$, we obtain a map $\tilde u/|\tilde u|: \ \partial\omega_R(x_0)\to S^1$, and 
define the degree of $u$ in $A_{r,R}(x_0)$ 
by 
\be\label{degree}
\bdeg=\deg\left(\frac{\tilde u}{|\tilde u|};\partial\omega_R\right).
\ee
Note that by the continuity of $g$, for $R$ small the complex phase difference of $\tilde u$ along $\Gamma_R$  is small (on the order of $R$.)  Thus, the winding of the phase around a boundary vortex occurs principally around the half-circle $\partial B_R(x_0)\cap\Omega$.
Let $g=e^{i\gamma}$, and $\gamma_0:=\gamma(x_0)$.
If we represent $u$ in polar coordinates $(\rho,\theta)$, centered at $x_0$ with $\rho=|x-x_0|$ and $\theta$ measured with respect to the positively oriented tangent line to $\Gamma$ at $x_0$, 
\be\label{polar}   u= f(\rho,\theta) e^{i\psi(\rho,\theta)}, \quad
    \text{with} \quad \psi= 2\bdeg\theta + \gamma_0+\phi(\rho,\theta),
\ee
and $\phi$ a smooth single-valued function in the annulus $A_{r,R}(x_0)$.
This is an essential difference between our boundary condition and the one studied in ferromagnetism \cite{Moser}, \cite{Ku}.  Here, the phase must make a complete cycle around a boundary vortex, while in the ferromagnetic models it is only required to make a half-turn at each defect.

The difference in cost between bulk and boundary vortices is contained in the following lower bound:

\begin{proposition}\label{vtxenergy}
Suppose $x_0\in\overline{\Omega}$, $0<r<R<r_0$, and assume that 
$\frac12 \le |u|\le 1$ in $A_{r,R}(x_0)$, $|u-g|\le \frac14$ on $\Gamma_R^\pm$, and there exists constants $C_1,C_2$ with $E_\eps(u)\le C_1|\ln\eps|$, and
\be\label{b1}
\frac{1}{2\eps^2}\int_{\omega_{\eps^\gamma}} (|u|^2-1)^2 \, dx
   + \lambda\int_{\Gamma_{\eps^\gamma}} |u-g|^2\, ds
        \le C_2
\ee
  Then there exists a constant $C$ such that:
\begin{enumerate}
\item[(a)] if $B_R(x_0)\cap\Gamma=\emptyset$, and 
$d=\deg\left(\frac{u}{|u|};\partial B_R(x_0)\right)$, then:
$$  \frac12\int_{A_{r,R}(x_0)} |\nabla u|^2\, dx \ge \pi d^2\ln\frac{R}{r} + C;
$$
\item[(b)] if $x_0\in\Gamma$, and $\bdeg$ is the degree of $u$ in $A_{r,R}(x_0)$ (defined as in \eqref{degree}), then
$$ \frac12\int_{A_{r,R}(x_0)} |\nabla u|^2\, dx \ge 2\pi \bdeg^2\ln\frac{R}{r} + C.
$$
\end{enumerate}
\end{proposition}

\begin{proof}[]
Conclusion (a) is proven in \cite{Struwe}, \cite{Struwe2}, so we may assume $x_0\in\Gamma$.  Write $u$ in the polar form \eqref{polar} in $A_{r,R}(x_0)$.   We first claim that there exists a constant $C_3$ for which
\be\label{m1}
|\phi|\le C_3\left(|u-g| + \rho\right), \qquad\text{on $\Gamma^\pm_{r,R}$.}
\ee
Indeed, writing $g=e^{i\gamma}$ and using the representation \eqref{polar} for $u$, we have
\begin{align*}
|u-g|^2 &= f^2 + 1 -2 f\cos(2\bdeg\theta+\gamma_0-\gamma+\phi)\\
&= (f-1)^2 + 2f\left(1-\cos(2\bdeg\theta+\gamma_0-\gamma+\phi\right)\\
&\ge 2f\left( 1-\cos(2\bdeg\theta+\gamma_0-\gamma+\phi)\right)\\
&\ge  1-\cos\phi\cos(2\bdeg\theta +\gamma_0-\gamma) + \sin\phi\sin(2\bdeg\theta +\gamma_0-\gamma),
\end{align*}
on $\Gamma^\pm_{r,R}$.  For all sufficiently small $R$, since $\Gamma$ is smooth, the arcs composing $\Gamma^\pm_{r,R}$ lie nearly along the tangent to $\Gamma$ at $x_0$, and hence 
$|1-\cos(2\bdeg\theta+\gamma_0-\gamma)|\le C\rho$ and $|\sin(2\bdeg\theta+\gamma_0-\gamma)|\le C\rho$ for constant $C$.  Thus, we have the estimate
$$  |u-g|^2 \ge 1-\cos\phi - C\rho\sin\phi\ge \frac12 \phi^2 - C\rho|\phi|
  \ge \frac14\phi^2 - C^2\rho^2, $$
which holds on $\Gamma^\pm_{r,R}$.  It follows that
$$  |\phi| \le 2\sqrt{|u-g|^2 + C^2\rho^2}\le C_3(|u-g|+\rho), $$
on $\Gamma^\pm_{r,R}$, as claimed.

The rest of the proof follows as in Proposition~5.6 of \cite{Moser}, except our representation \eqref{polar} differs from (5.31) of \cite{Moser} in the factor $2\bdeg$ appearing in the phase.  In this way, (5.32) of \cite{Moser} is modified to
\begin{align*}
|\nabla u|^2 &\ge f^2\left| 2\bdeg\nabla\theta + \nabla\phi\right|^2 \\
&\ge 4\frac{\bdeg^2}{\rho^2} + \left[
     \frac{4\bdeg^2}{\rho^2}(f^2-1) + \frac{4\bdeg}{\rho^2}\frac{\partial\phi}{\partial\theta} + f^2|\nabla\phi|^2\right].
\end{align*}
The first term on the right-hand side gives the desired lower bound, and the remaining terms may be estimated using exactly the computations in (5.34)--(5.39) in \cite{Moser}, replacing his $|f\cdot\nu|$ by $|u-g|$ throughout.
\end{proof}

\section{Locating the vortices}

We define the family of sets
$$  S_{\eps} =\left\{ x\in\overline\Omega: \  |u_\eps(x)| <\frac12 \ \text{or} 
\ |u_\eps(x)-g(x)|>\frac14\right\}.
$$
The following is a modification of Lemmas~3.1 and 3.2 of \cite{Struwe}:
\begin{lemma}\label{sballs}
There exists $N_0$ depending only on $\Omega$, $g$, and $h$, and points $p_{\eps,1},\dots,p_{\eps,I_\eps}\in S_\eps\cap\Omega$, $q_{\eps,1},\dots,q_{\eps,J_\eps}\in S_\eps\cap\Gamma$ such that 
\begin{enumerate}
\item[(i)]  $I_\eps + J_\eps\le N_0$;
\item[(ii)] 
$\{ B_{\eps}(p_{\eps,i}), B_{\eps^\alpha}(q_{\eps,j})\}_{1\le i\le I_\eps, 1\le j\le J_\eps}$ are mutually disjoint, and
\be\label{5cover}   S_\eps \subset \bigcup_{i=1}^{I_\eps}B_{5\eps}(p_{\eps,i})\cup 
    \bigcup_{j=1}^{J_\eps}B_{5\eps^\alpha}(q_{\eps,j}).  
\ee
\end{enumerate}
\end{lemma}

\begin{proof}[]
This is essentially the same as in \cite{Struwe}, who considered the case of Dirichlet boundary conditions, for which all of the ``bad balls'' have the same radius $\eps$.  We provide a sketch for completeness.  Let $y\in S_\eps$.  By Proposition~\ref{eta}, $E_\eps(u_\eps;B_{\eps^\gamma}(y)) > \eta|\ln\eps|$.  Applying Vitali's lemma to the collection $(B_{\eps^\gamma}(y))_{y\in S_\eps}$, there is a finite choice $y_1,\dots,y_N\in\overline{S}$ for which $(B_{\eps^\gamma}(y_i))_{i=1,\dots,N}$ are disjoint, and $(B_{5\eps^\gamma}(y_i))_{i=1,\dots,N}$ cover $S_\eps$.
Thus, by the upper bound \eqref{ub}
$$  N\eta|\ln\eps| \le \sum_{i=1}^N E_\eps(u_\eps; B_{\eps^\gamma}(y_i))
   \le E_\eps(u_\eps) \le K|\ln\eps|.
$$
In particular, $N$ is uniformly bounded independently of $\eps$.

Next, using the same argument as in \eqref{Fest}, there exists $r_\eps\in (\eps^\gamma,\eps^\beta)$ such that 
$$F(r_\eps)\le E(u_\eps;\omega_{\eps^\beta\setminus\eps^\gamma})/(\gamma-\beta),$$
so by Lemma~\ref{poho} we obtain the uniform estimate
$$  \frac{1}{2\eps^2}\int_{\omega_{r_\eps}(y_i)} (|u_\eps|^2-1)^2\, dx
         + \lambda\int_{\Gamma_{r_\eps}(y_i)} |u_\eps-g|^2\, ds
             \le C_7,
$$
for constant $C_7$ independent of $\eps, i=1,\dots,N$.

On the other hand, by the arguments employed in the proof of Lemma~\ref{eta}, there exists a constant $C_6$ (independent of $\eps$) such that if $B_{\eps}(y_i)\in \Omega$, 
$$  \frac{1}{2\eps^2}\int_{\omega_\eps(y_i)} (|u_\eps|^2-1)^2\, dx
   \ge C_6,
$$ 
while if $B_{\eps^{\alpha}}(y_i)\cap \Gamma\neq\emptyset$, 
$$ \frac{1}{2\eps^2}\int_{\omega_{\eps^\alpha}(y_i)} (|u_\eps|^2-1)^2\, dx
         + \lambda\int_{\Gamma_{\eps^\alpha}(y_i)} |u_\eps-g|^2\, ds
           \ge C_6.
$$
The conclusion then follows as in Lemma~3.2 of \cite{Struwe}:  by Vitali's lemma, there exist finite collections of points $(p_{\eps,i})_{i=1,\dots,I_\eps}$ in $\Omega$, $(q_{\eps,j})_{j=1,\dots,J_\eps}$ on $\Gamma$, satisfying (ii).  Finally, the cardinality of the sets is uniformly bounded, since
\begin{align*}
(I_\eps + J_\eps)C_6 & \le 
\sum_i^{I_\eps} \frac{1}{2\eps^2}\int_{\omega_\eps(p_i)} (|u_\eps|^2-1)^2\, dx
+\sum_j^{J_\eps}\left[ \frac{1}{2\eps^2}\int_{\omega_{\eps^\alpha}(q_j)} (|u_\eps|^2-1)^2\, dx
         + \lambda\int_{\Gamma_{r_\eps}(q_j)} |u_\eps-g|^2\, ds\right]\\
         &\le
\sum_{i=1}^N \frac{1}{2\eps^2}\int_{\omega_{r_\eps}(y_i)} (|u_\eps|^2-1)^2\, dx
         + \lambda\int_{\Gamma_{r_\eps}(y_i)} |u_\eps-g|^2\, ds\\
      &\le NC_7.
\end{align*}
\end{proof}

Next, we would like to follow \cite{Struwe} and \cite{BBH2} and prove a lower bound for the energy in small balls around the approximate vortices $p_{\eps,i},q_{\eps,j}$.  This may be done in a straightforward way in case $\Omega$ is a bounded domain, although it leads to different estimates depending on whether the vortex is located in $\Omega$ or on $\Gamma$.  A more serious complication arises when considering exterior domains $\Omega$, as we must handle the possibility that some vortices diverge to infinity as $\eps\to 0$.  From Lemma~\ref{sballs} we may nevertheless identify a finite number of balls, some fixed and some moving with $\eps$.  We summarize the construction in the following:

\begin{proposition}\label{vtxballs}
For any sequence of $\eps\to 0$, there is a subsequence $\eps_n\to 0$, a constant $\sigma_0>0$, finite collections of points $\{p_1,\dots,p_I\}\subset\Omega$, $\{q_1,\dots,q_J\}\subset\Gamma$, and a finite number of sequences, $(z_{k,n})_{n\in\mathbb{N}}\subset \Omega$ with $|z_{k,n}|\to\infty$ for each fixed $k=1,\dots,K$, so that for any $\sigma\in (0,\sigma_0)$ and for all $n\in\mathbb{N}$,
$$   \mathcal{S}_\sigma:=\{ B_\sigma(p_i)\}_{i=1,\dots,I} \cup
     \{ B_\sigma(q_j)\}_{j=1,\dots,J} \cup
     \{ B_\sigma(z_{k,n})\}_{k=1,\dots,K}
$$
is a collection of mutually disjoint sets which cover $S_{\eps_n}$.
\end{proposition}

\begin{proof}[]
In case $\Omega$ is bounded, the number of divergent sequences $K=0$.    In case $\Omega$ is unbounded and certain sequence $|p_{\eps_n,i}|\to\infty$, we choose $z_{1,n}$ to be any one of those $p_{\eps_n,i}$.   If there is a different sequence $p_{\eps_n,j}$ with
$|p_{\eps_n,j}|\to\infty$ but $|z_{1,n}-p_{\eps_n,j}|\not\to 0$, we let $z_{2,n}=p_{\eps_n,j}$ for that $j$.  As the number of sequences is finite, this process will end with the definition of a finite number of sequences $(z_{k,n})_n$, and for any $i=1,\dots,I$, either the sequence $p_{\eps_n,i}$ remains bounded or there exists $k\in\{1,\dots,K\}$ for which $|z_{k,n}-p_{\eps_n,i}|\to 0$.  By passing to a further subsequence, each of the bounded sequences converge to the  $p_i\in\Omega$ or $q_j\in\Gamma$.  The constant $\sigma_0$ may be chosen smaller than half the distance between any pair of the $p_i,q_j$, and smaller than $\frac12\liminf_{n\to\infty}|z_{k,n}-z_{\ell,n}|>0$, for any $k\neq\ell$.  As $\sigma$ is fixed, $\mathcal{S}_\sigma$ will eventually contain $S_{\eps_n}$ for $n$ large enough.
\end{proof}

Since $\mathcal{S}_\sigma$ covers $S_{\eps_n}$, $|u_{\eps_n}|\ge\frac12$ on $\partial\mathcal{S}_\sigma$, and hence we may define degrees associated to each ball in $\mathcal{S}_\sigma$.
\begin{align*}  d_i:&=\deg(u_{\eps_n};\partial B_\sigma(p_i)), \quad i=1,\dots,I,  \\
\bdeg_j:&=\deg(u_{\eps_n};\partial B_\sigma(q_j)), \quad j=1,\dots,J,  \\
\tilde d_k:&=\deg(u_{\eps_n};\partial B_\sigma(z_{k,n})), \quad k=1,\dots,K, \ n\in\mathbb{N}.
\end{align*}
We recall that in the case of the boundary vortices, the degree is defined in the sense of \eqref{degree}.  Although the weak anchoring condition is not a Dirichlet condition, the total degree of minimizers is still given by the degree of the boundary value.

\begin{lemma}\label{degreelem}
Let $u_{\eps_n}$, $d_i,\bdeg_j$, $\tilde d_k$ be as above.
Then we have:
\begin{enumerate}
\item[(a)] For Problem I,
$\displaystyle\DDD:=\deg(g;\Gamma) = \sum_{i=1}^I d_i + \sum_{j=1}^J  \bdeg_j.$
\item[(b)]  For Problem II, 
$-\displaystyle\DDD = \sum_{i=1}^I d_i + \sum_{j=1}^J  \bdeg_j.$  
\item[(c)]  For Problem III, 
$-\displaystyle\DDD = \sum_{i=1}^I d_i + \sum_{j=1}^J  \bdeg_j
      +\sum_{k=1}^K \tilde d_i.$      
\end{enumerate}
\end{lemma}

\begin{proof}[]
First, consider Problem I, with $\Omega$ simply connected and $\Gamma=\partial\Omega$.
Let $\tilde\Omega=\Omega\setminus\left[\bigcup_{j=1}^J \omega_\sigma(q_j)\right]$, and $\tilde\Gamma=\partial\tilde\Omega$.
Fix $\sigma$ small enough that $\partial\omega_\sigma(q_j)\cap\Gamma$ consists of exactly two points for each $j=1,\dots,J$.
We recall the definition of the degree $\bdeg_j$:  Since $|u_{\eps_n}-g|<\frac14$ on the two endpoints of $\partial\omega_\sigma(q_j)\cap\Gamma$, we may define a Lipshitz extension $\tilde u_{\eps_n}$  of $u_{\eps_n}$to $\Gamma_\sigma(q_j)$ for which both $|\tilde u_{\eps_n}-g|\le\frac 12$ for each $j=1,\dots,J$.   (On $\Gamma\setminus\cup_j \Gamma_\sigma(q_j)$, we take $\tilde u_{eps_n}=u_{\eps_n}$.)  Since $|\tilde u_{\eps_n}-g|\le\frac12$ on all of $\Gamma$, it follows that $\deg(\tilde u_{\eps_n};\Gamma)=\deg(g;\Gamma)=\mathcal{D}$.  

Consider now the simple closed curve $\tilde \Gamma:=\partial\tilde\Omega$.  We have $|u_{\eps_n}|\ge \frac12$ on $\tilde\Gamma$, and so its degree is well-defined, and
\begin{align*}
\deg(u_{\eps,n};\tilde\Gamma) &=
         \frac1{2\pi}\int_{\Gamma\setminus\cup_j \Gamma_\sigma(q_j)}
               \frac{(iu_{\eps_n},\partial_\tau u_{\eps_n})}{|u_{\eps_n}|^2} ds
               + \frac1{2\pi}\int_{\partial\omega_\sigma(q_j)\cap\Omega}
               \frac{(iu_{\eps_n},\partial_\tau u_{\eps_n})}{|u_{\eps_n}|^2} ds \\
               &=  \frac1{2\pi}\int_{\Gamma}
               \frac{(i\tilde u_{\eps_n},\partial_\tau \tilde u_{\eps_n})}{|u_{\eps_n}|^2} ds
               - \frac1{2\pi}\int_{\partial\omega_\sigma(q_j)}
               \frac{(i\tilde u_{\eps_n},\partial_\tau \tilde u_{\eps_n})}{|u_{\eps_n}|^2} ds \\
&= \deg(\tilde u_{\eps_n},\Gamma) - \sum_{j=1}^J \bdeg_j \\
&= \mathcal{D} - \sum_{j=1}^J \bdeg_j,
\end{align*}
where we have used the fact that the arcs $\Gamma_\sigma(q_j)$ are common to both integrals.  Finally, the vortices $p_i$ are contained inside $\tilde \Gamma$, and hence $\deg(u_{\eps,n};\tilde\Gamma)=\sum_i d_i$, and the assertion (a) follows.

For Problems II and III, we make a similar construction, but now the arcs $\Gamma_\sigma(q_j)$, while common to the integrals over $\Gamma$ and $\partial\omega_\sigma(q_j)$ are oriented in the opposite sense.
Therefore,
\begin{align*}
\deg(u_{\eps,n};\tilde\Gamma)                
     &=  \frac1{2\pi}\int_{\Gamma}
               \frac{(i\tilde u_{\eps_n},\partial_\tau \tilde u_{\eps_n})}{|u_{\eps_n}|^2} ds
               + \frac1{2\pi}\int_{\partial\omega_\sigma(q_j)}
               \frac{(i\tilde u_{\eps_n},\partial_\tau \tilde u_{\eps_n})}{|u_{\eps_n}|^2} ds \\
&= \deg(\tilde u_{\eps_n},\Gamma) + \sum_{j=1}^J \bdeg_j \\
&= \mathcal{D} + \sum_{j=1}^J \bdeg_j.
\end{align*}
In Problem II, the vortices $p_i$ lie outside of $\tilde\Gamma$, while the degree of $u_{\eps_n}$ is zero on the outside boundary $\partial\Omega_1$. Thus,
$$  0=\deg(u_{\eps_n};\tilde\Gamma) + \sum_{i=1}^I d_i = \mathcal{D} + \sum_{j=1}^J \bdeg_j +\sum_{i=1}^I d_i,
$$
and (b) must hold.  The result (c) for Problem III follows in the same way, as $u_{\eps_n}$ has degree zero outside of a circle of radius $R_n$ which is sufficiently large to enclose the moving vortices $z_{k,n}$.
\end{proof}

Starting with the lower bound on annuli proven in Proposition~\ref{vtxenergy}, and arguing as in Proposition~3.3 of \cite{Struwe}, (or by the vortex-ball method of Jerrard \cite{Jerrard} or Sandier \cite{Sandier},) we may obtain the following lower bound on the energy inside the set $\mathcal{S}_\sigma$:

\begin{lemma}\label{LB}
There exists a constant $C$, independent of $\eps_n,\sigma$ such that:
\begin{align*}
E_{\eps_n}\left(u_{\eps_n}; B_\sigma(p_i)\right)&\ge
    \pi |d_i|\ln\left(\frac{\sigma}{\eps_n}\right) - C, \quad i=1,\dots,I,\\
E_{\eps_n}\left(u_{\eps_n}; B_\sigma(q_j)\right)&\ge
  2\pi |\bdeg_j|\ln\left(\frac{\sigma}{\eps_n^\alpha}\right) - C, \quad     
     j=1,\dots,J,\\
E_{\eps_n}\left(u_{\eps_n}; B_\sigma(q_j)\right)&\ge
    \pi |\tilde d_k|\ln\left(\frac{\sigma}{\eps_n}\right) - C, \quad k=1,\dots,K.    
\end{align*}
\end{lemma}
As an immediate consequence, there exists a constant $C_1(\sigma)$ such that
\be\label{lb}
E_{\eps_n}\left(u_{\eps_n}; \mathcal{S}_\sigma\right) \ge
\pi\left[  \sum_{i=1}^I |d_i| + \sum_{j=1}^J 2\alpha |\bdeg_j|
      +\sum_{k=1}^K |\tilde d_i|\right] |\ln\eps| - C_1(\sigma).
\ee

Denote by
$$  \Sigma:= \{p_i\}_{i=1,\dots,I} \cup  \{q_j\}_{j=1,\dots,J}. $$
Comparing with the upper bound \eqref{ub}, we obtain the following:

\begin{theorem}\label{limthm}
For any sequence of $\eps\to 0$, there exists a subsequence  $\eps_n\to 0$ such that:
\begin{enumerate}
\item[(a)]
The sets $S_{\eps_n}$ are uniformly bounded; thus $K=0$.
\item[(b)]  
For all $0<\alpha<\frac12$, the vortices occur on $\Gamma$ only; $I=0$.  Each $|\bdeg_j|=1$  and has the same sign.
\item[(c)]
For all $\frac12<\alpha\le 1$, all vortices lie in $\Omega$; $J=0$.
Each $|d_i|=1$  and has the same sign.
\item[(d)] For $\alpha=\frac12$, both boundary and interior vortices are possible.  Each $|d_i|,|\bdeg_j|=1$ and has the same sign.
\item[(e)]  For any $0<\alpha\le 1$ and all $\ell\ge 0$, $u_{\eps_n}\to u_*$ in $C^\ell_{loc}(\overline{\Omega}\setminus\Sigma)$, where $u_*$ is a smooth harmonic map with values in $S^1$.  Moreover,
$u_*=g$ on $\Gamma\setminus\Sigma$, and there exists $\phi_*\in\R$ for which 
\be\label{ustarlimit}
u_*(x)\to e^{i\phi_*}\qquad\text{ as $|x|\to\infty$.}
\ee
\end{enumerate}
\end{theorem}

We note that in the case $\frac12<\alpha\le 1$, $u_{\eps,n}\to g$ uniformly on $\Gamma$.

\begin{proof}[]
Comparing the lower bound \eqref{lb} with the upper bound \eqref{ub}, we have
$$   \sum_{i=1}^I |d_i| + \sum_{j=1}^J 2\alpha |\bdeg_j|
      +\sum_{k=1}^K |\tilde d_i| \le \min\{2\alpha,1\} \DDD.  $$
When $0<\alpha<\frac12$, we have
$$  2\alpha\DDD + (1-2\alpha)\left[
         \sum_{i=1}^I |d_i| +\sum_{k=1}^K |\tilde d_i|\right] \le 2\alpha \DDD,$$
and hence $d_i,\tilde d_k=0$ for all $i,k$.  In addition, $\sum_{j=1}^J  |\bdeg_j|= \DDD= \left|\sum_{j=1}^J  \bdeg_j\right|$, and hence each
$\bdeg_j$ must have the same sign (or vanish.)
In case $\frac12<\alpha\le 1$, the same argument produces the opposite result:  each $\bdeg_j=0$, and the nonzero $\bdeg_i,\tilde \bdeg_k$ all have the same sign.  When $\alpha=\frac12$, we may only conclude that the nonzero $d_i,\bdeg_j,\tilde d_k$  all have the same sign.

In any case, the lower bound \eqref{lb} and upper bound \eqref{ub} together imply that there exists a constant $C_2(\sigma)$ for which
\be\label{sigmabound}
  E_{\eps_n}(u_{\eps_n}; \Omega\setminus\mathcal{S}_\sigma) \le C_2(\sigma).  
\ee

We next claim that, in the case that $\Omega$ is an exterior domain, $\tilde d_k=0$ for all $k$.  Suppose not, so 
$\tilde d:=\left|\sum_{k=1}^K  \tilde d_k\right|= \sum_{k=1}^K  |\tilde d_k|\ge 1.$
By Theorem~\ref{exterior}, each $u_{\eps_n}\to e^{i\phi_0}$, as $|x|\to\infty$.  
Thus, there exists $R_{3,n}$ for which $\deg(u_{\eps_n};\partial B_{R_{3,n}})=0$.  
Since each $|z_{k,n}|\to\infty$, there exists $R_{2,n}\to \infty$ so that $|z_{k,n}|>2R_{2,n}$ for each $k=1,\dots,K$.  Note that $|\deg(u_{\eps_n};\partial B_{R_{2,n}})|=\tilde d\neq 0$. 
 Finally, we may choose a fixed radius, $R_1>0$ for which all the $|p_i|,|q_j|< \frac12 R_1$.  In particular, $|u_{\eps_n}|\ge\frac12$ on $\overline{B_{R_{2,n}}}\setminus B_{R_1}$, and thus $|\deg(u_{\eps_n};\partial B_r)|=\tilde d\neq 0$ for all $r\in [R_1, R_{2,n}]$, for all $n$.  But then we obtain the lower bound,
 $$  E_{\eps_n}(u_{\eps_n}; \Omega\setminus\mathcal{S}_\sigma)
 \ge E_{\eps_n}(u_{\eps_n}; B_{R_{2,n}}\setminus B_{R_1})
    \ge C_3\ln \frac{R_{2,n}}{R_1}\to \infty, $$
which contradicts the upper bound \eqref{sigmabound}.  In conclusion, $\tilde d_k=0$ for all $k=1,\dots,K$ as claimed.

The remainder of the proof follows \cite{BBH2}.  Indeed, the fact that none of the degrees $d_i,\bdeg_j,\tilde d_k=0$ follows Step~1 in the proof of Theorem~VI.2 of \cite{BBH2}, and the rest of that Theorem holds as above, except that in exterior domains we expect negative rather than postive degrees.  Once we have established that $\tilde d_k=0$ is not possible, it follows that $K=0$ and the set $S_{\eps_n}$ must be uniformly bounded.  The convergence to a harmonic map, outside of the singular set $\Sigma$, is proven first in $W^{1,2}_{loc}$ (see \cite{Struwe}), and then in stronger norms using \cite{BBH1}.
To prove \eqref{ustarlimit}, since the singular sets $S_{\eps_n}\subset B_R$ are uniformly bounded, we conclude from \eqref{sigmabound} that
$$  \int_{\R^2\setminus B_R} |\nabla u_{\eps_n}|^2\, dx \le C_2(\sigma).
$$
Passing to the limit $u_{\eps_n}\wto u_*=e^{i\varphi_*}$, we obtain the bound $\int_{\R^2\setminus B_R} |\nabla \varphi_*|^2\, dx \le C_2(\sigma)$.  Since $\varphi_*(x)$ is harmonic in $\R^2\setminus B_R$, we conclude that infinity is a removable singularity for $\varphi_*$ and thus $\varphi_*(x)\to\phi_*$ for a constant $\phi_*\in\R$.
\end{proof}

\begin{remark}\label{canon} \rm
As in \cite{BBH2}, \cite{Riviere} the limit is described in terms of canonical harmonic maps, with the observation that the structure of the singularity at a boundary vortex is modified as follows:
$$  u_*(z)=\prod_{i=1}^I\left[ \frac{z-p_i}{ |z-p_i|}\right]^{d_i}\cdot
    \prod_{j=1}^J\left[ \frac{z-q_i}{ |z-q_i|}\right]^{2\bdeg_i}\, e^{i\xi(z)},
$$
with degrees $d_i,\bdeg_j=\pm 1$, and $\Delta\xi=0$ in $\Omega$.
\end{remark}

We note that, thanks to Theorem~\ref{limthm}, we have verified statements (a)--(c) of Theorem~\ref{mainthm}.  The remaining parts of Theorem~\ref{mainthm}, as well as the more detailed conclusions of Theorem~\ref{LQthm}, rely on the study of the Renormalized Energies for each problem, and will be proven in the following section.

\section{Renormalized Energies}\label{RNsec}

To locate the vortices of energy minimizers we use the Renormalized Energy as in \cite{BBH2}.  We proceed separately for each of the three problems considered above, defining harmonic conjugate functions suitable for each.  As we are mostly interested in giving some qualitative interpretation to the results for weak coupling in some specific geometries, we omit the (voluminous) details involved in connecting the Renormalized Energy to the Ginzburg-Landau minimizers; the details follow the same lines as those in \cite{BBH2} or \cite{Riviere}.  As in either of these references, one may derive a rigorous asymptotic expansion of the energy of minimizers of the form:
\be\label{asympexp}
E_\eps(u_\eps) = I (\pi|\ln\eps|+Q_\Omega) +  J (2\pi\ln\lambda + Q_\Gamma) + W(p_1,\dots,p_I,q_1,\dots,q_J)  +o(1),
\ee 
where $Q_\Omega$, $Q_\Gamma$ are constants (representing the energy of vortex cores inside $\Omega$ or on $\Gamma$.)  Here
 $W: \ \Omega^{\mathcal D}\times \Gamma^{\mathcal{D}}\to \R$ is the Renormalized Energy, whose definition and properties we will discuss in more detail below.

\paragraph{Problem I.}
We begin with Problem I in the bounded simply connected domain $\Omega$ with $\Gamma=\partial\Omega$.  This is the case which is most like the familiar Dirichlet case studied in \cite{BBH2}. We assume the total degree $\mathcal D>0$, and thus each vortex has degree $+1$. Let $\Phi_I(x)=\Psi_I(x;\{p_i\},\{q_j\})$ solve
\be\label{PhI}
\left.
\begin{gathered}
\Delta\Phi_I=2\pi\sum_{i=1}^I \delta_{p_i}(x), \quad\text{in $\Omega$,}\\
\frac{\partial\Phi_I}{\partial\nu}= g\times g_\tau - 2\pi\sum_{j=1}^J \delta_{q_j}(x), \quad\text{on $\Gamma$}.
\end{gathered}
\right\}
\ee
We note that either one of the collections $\{p_i\}$ or $\{q_j\}$ may be empty:  indeed, by Theorem~\ref{limthm}, the former will occur for $\alpha\in (0,\frac12)$ and the latter for $\alpha>\frac12$, and the two collections may only coexist in evaluating the energy of minimizers of $E_\eps$ when $\alpha=\frac12$.

The Renormalized Energy corresponding to the problem I is (see \cite{Riviere},)
\be\label{RNen}  W_I(\{p_i,d_i\},\{q_j,\bdeg_j\}):=
   \lim_{\rho\to 0}\left(
     \frac12\int_{\Omega\setminus\mathcal{S}_\rho} 
     |\nabla\Phi_I(x;\{p_i\},\{q_j\})|^2\, dx
        -\pi \left[ I   + 2 J \right] \ln\frac{1}{\rho}\right).
\ee
By proving sharp upper and lower bounds as in \cite{BBH2}, it may be shown that the limiting singularities of the sequence of minimizers $u_{\eps_n}$ minimize $W(\{p_i,d_i\},\{q_j,\bdeg_j\})$ within the topological and energy constraints given by the weak anchoring condition $g$ and the choice of $\alpha\in (0,1]$.  Namely, if $0<\alpha<\frac12$, by Theorem~\ref{limthm},
$I=0$ and $J=\DDD$, and $W$ depends only on $\{q_1,\dots,q_{\mathcal{D}}\}\subset\Gamma$, with each degree $\bdeg_j=\pm 1$ the same and determined as in Lemma~\ref{degreelem}, according to the problem under consideration.  On the other hand, if $\alpha>\frac12$, then $I=\DDD$, $J=0$, and $W$ depends only on $\{p_1,\dots,p_{\mathcal{D}}\}\subset\Omega$, with degrees $d_i=\pm 1$ all identical, again determined by Lemma~\ref{degreelem}.  
When $\alpha=\frac12$, $I+J=\DDD$ and the minimization of $W$ must be performed among all combinations of $\DDD$ vortices on $\Gamma$ and inside $\Omega$.  However, we note that in that case $\ln\lambda=\frac12|\ln\eps| + \ln K$, the energy expansion \eqref{asympexp}
takes the form
$$  E_\eps(u_\eps) = \pi (I+J) |\ln\eps|+ \left\{IQ_\Omega + J ( Q_\Gamma + \ln K) + W(p_1,\dots,p_I,q_1,\dots,q_J)\right\}  +o(1).  $$
At highest order, boundary and interior vortices have the same unit cost, but by making $K>0$ very small or very large the choice of boundary or interior vortices may become more favorable, by either favoring or penalizing the coefficient of $J$ in the energy expansion, nullifying any advantage one has over the other in either the core cost $Q_\Omega, Q_\Gamma$ or in the minimum value of the Renormalized Energy $W$.  Thus, by taking $K>0$ very small, we may ensure that all vortices reside on $\Gamma$, while for $K>0$ sufficiently large they must be found inside $\Omega$.  This completes the proof of Theorem~\ref{mainthm} for Problem I.

\smallskip

\paragraph{Problem II.}
As pointed out in I.2 of \cite{BBH2}, the evaluation of the Renormalized Energy in multiply connected domains with Dirichlet boundary values on each component of $\partial\Omega$ is tricky, and our problem II exhibits these same difficulties.   It turns out that we may still obtain an explicit representation of the Renormalized Energy in the special case
$$  \Omega = B_R(0)\setminus\overline{B_1(0)}, \qquad
  g=u|_{\partial B_1(0)} = e^{i\DDD\theta},
$$
with $\DDD\in\mathbb{N}$.
We recall that in Problems II and III, the vortices have degree $-1$, and begin by introducing a conjugate harmonic problem
in the bounded annular domain $\Omega=\Omega_1\setminus\overline{\Omega_0}$, in analogy with \eqref{PhI}:  let $\Phi_{II}=\Phi_{II}(x;\{p_i\},\{q_j\})$ solve
\be\label{PhII}
\left.
\begin{aligned}
\Delta\Phi_{II}&=-2\pi\sum_{i=1}^I \delta_{p_i}(x), \quad\text{in $\Omega$,}\\
\frac{\partial\Phi_{II}}{\partial \nu}&= g\times g_\tau - 2\pi\sum_{j=1}^J \delta_{q_j}(x) \\
&=\DDD - 2\pi\sum_{j=1}^J \delta_{q_j}(x)
, \quad\text{on $\Gamma$}\\
\frac{\partial\Phi_{II}}{\partial \nu}&=0\quad\text{on $\partial\Omega_1$}.
\end{aligned}
\right\}
\ee
While $\Phi_{II}$ is an ingredient in the Renormalized Energy, some adjustment must be made to match the Dirichlet boundary conditions on both components of $\partial\Omega$.

We introduce auxilliary problems, with a single vortex located on the negative $x_1$-axis:  for an interior vortex at $p=(-t,0)$, $1<t<R$, let $\Phi^t$ solve
\begin{equation}\label{Phit}
\left.
\begin{aligned}
-\Delta \Phi^t &= 2\pi\delta_{(-t,0)}(x), \quad\text{inside $\Omega$},\\
\frac{\partial\Phi^t}{ \partial\nu} &= 1, \quad\text{on $\Gamma=\partial B_1(0)$}, \\
\frac{\partial\Phi^t}{ \partial\nu} &= 0, \quad\text{on $\Gamma=\partial B_R(0)$}.
\end{aligned}\right\}
\end{equation}
For a single vortex  at the point $p=(-1,0)\in\Gamma$, we define $\Phi^1$ as the solution of:
\begin{equation}\label{Phi1}
\left.
\begin{aligned}
-\Delta \Phi^1 &= 0, \quad\text{inside $\Omega$},\\
\frac{\partial\Phi^1}{ \partial\nu} &= 1-2\pi\delta_{(-1,0)}, \quad\text{on $\Gamma=\partial B_1(0)$}, \\
\frac{\partial\Phi^1}{ \partial\nu} &= 0, \quad\text{on $\Gamma=\partial B_R(0)$}.
\end{aligned}\right\}
\end{equation}
Each is unique up to an additive constant; we choose that constant so that
$\int_\Gamma \Phi^t\, ds =0$, for each $t\in [1,R)$.
The basic building blocks for the singular harmonic map come from these auxilliary problems; we begin by proving:
\begin{lemma}\label{vt}
For each $t\in [1,R)$, there exists an $S^1$-valued harmonic map
$v_t\in H^1_{loc}(\overline{\Omega}\setminus \{(-t,0)\})$ such that
\begin{gather*}
(iv_t, \nabla v_t) = -\nabla^\perp \Phi^t, \quad
   \text{in $\Omega\setminus \{(-t,0)\}$},\\
 v_t =1 \quad\text{on $\partial B_R(0)$,}\\
 v_t= e^{i\theta}\quad \text{on $\partial B_1(0)\setminus\{(-t,0)\}$}.
\end{gather*}
\end{lemma}
Note that the last condition holds on all of $\partial B_1(0)$ in case $t\neq 1$.

\begin{proof}[]
First, define $\tilde\Omega_\eta=\Omega\setminus B_\eta(-t,0)$.  
We first consider the case that $t\in (1,R)$, and thus $B_\eta(-t,0)\subset\Omega$ (for $\eta$ sufficiently small).  Since $V:=\nabla^\perp\Phi^t$ is irrotational in $\tilde\Omega_\eta$ for any $\eta$, there exists (generally multivalued) $\phi\in H^1_{loc}(\overline{\Omega}\setminus \{(-t,0)\})$ for which we may locally represent $\nabla^\perp\Phi^t=-\nabla \phi$ as a gradient.  Since the equation \eqref{Phit} implies that 
$$\int_{\partial B_\eta(-t,0)} V\cdot\tau\,ds = \int_{\partial B_\eta(-t,0)} \frac{\partial\Phi^t}{\partial\nu}\,ds=-2\pi,\quad
\int_{\partial B_1(0)} V\cdot\tau\,ds = \int_{\partial B_1(0)} \frac{\partial\Phi^t}{\partial\nu}\,ds=2\pi,
$$
we may lift $\phi$ to a single-valued $S^1$-valued map $v_t:=e^{i\phi}$, with
$(iv_t, \nabla v_t) = -\nabla^\perp \Phi^t$ in $\overline{\Omega}\setminus \{(-t,0)\}$.
Using the boundary condition for $\Phi^t$ we may obtain boundary behavior for $v_t$.
On $\partial B_1(0)$, 
$(iv_t,\partial_\tau v_t)= \partial_\nu\Phi^t =1$ (with counterclockwise orientation), and hence we may choose the constant of integration when defining $v_t$ such that $v_t=e^{i\theta}$ on $\partial B_1(0)$.
Similarly, on $\partial B_R(0)$, we have $(iv_t,\partial_\tau v_t)=0$, and we conclude that $v_t$ is a constant of modulus one on $\partial B_R(0)$.

In the case $t=1$, the vortex lies on the inner boundary $\Gamma$, so the inner component of the boundary $\partial\tilde\Omega_\eta$ is composed of two circular arcs.  By the equation \eqref{Phit}, it follows that $\int_{\partial\tilde\Omega_\eta} V\cdot\tau\, ds=0$, and in this case the above argument actually yields 
a single-valued $\phi\in H^1(\tilde\Omega_\eta)$ for each $\eta$, and thus lifts to the $S^1$-valued map $v_t:=e^{i\phi}$ in $\overline{\Omega}\setminus \{(-1,0)\}$.  
Furthermore, arguing as in the previous case, we obtain the boundary value $v_t|_{\partial B_R(0)}$ is constant, while
$v_t=e^{i\theta}$ on $\partial B_R(0)\setminus\{(-1,0)\}$.

It remains to identify the constant value $v|_{\partial B_R(0)}$.
Let $\eta>0$, $\mathcal{N}_\eta$ an $\eta$-neighborhood of the negative $x_1$-axis, and $\hat\Omega_\eta=\Omega\setminus \mathcal{N}_\eta$, which is symmetric with respect to the $x_1$-axis and simply connected for all $\eta<1$.
We observe that $\Phi^t$ is even in $x_2$, for any $t\in [1,R)$, and so
$\partial_{x_1}\Phi^t$ is even in $x_2$, while $\partial_{x_2}\Phi^t$ is
odd in $x_2$.  As $\hat\Omega_\eta$ is simply connected, $\phi$ is single-valued there, and $\partial_{x_1}\phi=\partial_{x_2}\Phi^t$ is odd in $x_2$ while $\partial_{x_2}\phi=-\partial_{x_1}\Phi^t$ is even in $x_2$.  Hence, there is a choice of constant of integration for which $\phi$ is odd in $x_2$.
In particular, $\phi(x_1,0)=0$ for $x_1\in [1,R]$.  Since $v_t=e^{i\phi}$ is constant on $\partial B_R(0)$, we conclude that $v_t=1$.
\end{proof}

From Lemma~\ref{vt} we can see exactly how the position of the vortices affects the boundary condition imposed by the conjugate function $\Phi_{II}$.  Write each of the vortices in polar coordinates (in complex notation), but measuring the angle from $\pi$, $p_i=|p_i|e^{i(\pi-a_i)}$, $q_j=|q_j|e^{i(\pi-b_i)}$.

\begin{lemma}\label{v}
There exists an $S^1$-valued harmonic map
$v\in H^1_{loc}(\overline{\Omega}\setminus \{p_1,\dots,p_I,q_1,\dots,q_J\})$ such that
\begin{gather*}
(iv, \nabla v) = -\nabla^\perp \Phi_{II}, \quad
   \text{in $(\overline{\Omega}\setminus \{p_1,\dots,p_I,q_1,\dots,q_J\}$},\\
  v= e^{i\theta}\quad 
   \text{on $\partial B_1(0)\setminus\{q_1,\dots,q_J\}$},\\
  v =e^{-i(a_1+\cdots+a_I+b_1+\cdots+b_J)} \quad\text{on $\partial B_R(0)$.}\\
\end{gather*}
\end{lemma}

\begin{proof}[]
For each $i$, define (using complex notation for $z=x_1+ ix_2\in\overline{\Omega}$,) $\tilde v_{p_i}(z):= e^{-ia_i}v_{|p_i|}(e^{ia_i}z)$, using $t=|p_i|$ in $v_t$ from Lemma~\ref{vt}.  Since $\nabla v_{p_i}(z)= (\nabla v_{|p_i|})(e^{ia_i}z)= -\nabla^\perp\Phi^{|p_i|}(e^{ia_i}z)$, the function $\tilde\Phi_i(z):= \Phi^{|p_i|}(e^{ia_i}z)$ merely rotates problem \eqref{Phit} by angle $a_i$:
$$  -\Delta\tilde\Phi_i = 2\pi\delta_{p_i}, \ \text{in $\Omega$,}, \quad
     \partial_\nu\tilde\Phi_i|_{\partial B_1(0)}=1, \quad
     \partial_\nu\tilde\Phi_i|_{\partial B_R(0)}=0.
$$
Similarly, for each boundary vortex $q_j$, define 
$\hat v_{q_j}(z):=e^{-ib_j}v_{1}(e^{ib_j}z)$.  Then, 
$\nabla\hat v_{q_j}(z)=-\nabla^\perp\Phi^1(e^{ib_j}z)$, and defining
$\hat \Phi_j(z):=\Phi^1(e^{ib_j}z)$ is a rotation of problem \eqref{Phi1} by angle $b_j$,
$$  -\Delta\hat\Phi_j=0, \ \text{in $\Omega$,} \quad
   \partial_\nu\tilde\Phi_i|_{\partial B_1(0)}=1-2\pi\delta_{q_j}, \quad
     \partial_\nu\tilde\Phi_i|_{\partial B_R(0)}=0.
$$
In particular, we recover 
  $\Phi_{II}=\sum_{i=1}^I \tilde\Phi_i +\sum_{j=1}^J \hat\Phi_j$.
Now define
$$   v:=\left[\prod_{i=1}^I \tilde v_{p_i}\right]
    \left[\prod_{j=1}^J \hat v_{q_j}\right].
$$
Then, it is straightforward to verify that $v\in H^1_{loc}(\overline{\Omega}\setminus\{p_1,\dots,p_I,q_1,\dots q_j\}; S^1)$, $v$ is a harmonic map,
and
$(iv, \nabla v) = -\nabla^\perp \Phi_{II}$ in $\overline{\Omega}\setminus\{p_1,\dots,p_I,q_1,\dots q_j\}$.  Moreover, $v|_{\partial B_1(0)}=e^{i\theta}$ (as each of the rotations leaves $e^{i\theta}$ invariant), while at the other boundary component the constants superimpose,
$v|_{\partial B_R(0)} = e^{i(a_1+\cdots+a_I+b_1+\cdots+b_J)}$.
\end{proof}

To obtain the correct boundary condition $u|_{\partial B_R(0)}=1$ we must adjust the singular harmonic map $v$ by adding a harmonic function to the phase.  As in \cite{BBH2}, this is where the capacity of the annular domain $\Omega$ enters into the calculation of the energy.  Let $\psi\in H^1(\Omega;\R)$ denote the (unique) minimizer of the Dirichlet energy $\int_\Omega  |\nabla \psi|^2$, among functions satisfying 
$\psi|_{\partial B_1(0)}=0$ and $\psi|_{\partial B_R(0)}=1$.  The minimum energy 
$$  \int_\Omega |\nabla\psi|^2\, dx = \text{cap}_{B_R}(B_1)=\frac{2\pi}{ \ln R}, $$
gives the capacity of the hole $B_1(0)$ relative to the domain $B_R(0)$.  If we then define
$$  u(z) = v(z) e^{i(a_1+\cdots+a_I+b_1+\cdots+b_J)\psi(z)}, $$
then it is easy to verify that $u$ is an $S^1$-valued singular harmonic map in $\overline{\Omega}\setminus\{p_1,\dots,p_I,q_1,\dots q_j\}$, which satisfies the desired boundary conditions, $u|_{\partial B_1(0)}=e^{i\theta}$ and $u|_{\partial B_R(0)}=1$.  Moreover, by the construction of $v$ in Lemma~\ref{v}, $u$ is a canonical harmonic map; that is, it satisfies the structural equation given in Remark~\ref{canon}.  

Let $\beta=a_1+\cdots+a_I+b_1+\cdots+b_J$.
$$   (iu,\nabla u) = (iv, \nabla v) + \beta\nabla\psi=
        -\nabla^\perp\Phi_{II} + \beta\nabla\psi.  $$
Since $|u|=1$ in $\Omega_\rho$, we derive the following expansion of the Dirichlet energy,
\begin{align}\nnn
  \int_{\Omega_\rho} |\nabla u|^2\, dx
  &= 
  \int_{\Omega_\rho} \left[  (iu,\partial_{x_1}u)^2 + (iu,\partial_{x_2}u)^2\right] dx \\
  \nnn
  &=   \int_{\Omega_\rho} \left[ |\nabla^\perp \Phi_{II}|^2 
  + \beta^2 |\nabla\psi|^2 
     - 2\beta\nabla^\perp\Phi_{II}\cdot\nabla \psi\right] dx \\
     \nnn
&=    \int_{\Omega_\rho} |\nabla^\perp \Phi_{II}|^2
              + \frac{2\pi}{\ln R}\beta^2 
              + \int_{\partial\Omega_\rho} \Phi_{II}\nabla\psi\cdot\tau\, ds
                + (\rho^2) \\
                \label{enex}
&= \int_{\Omega_\rho} |\nabla^\perp \Phi_{II}|^2
              + \frac{2\pi}{\ln R}\beta^2  + (\rho^2),
\end{align} 
as $\psi$ is constant on $\partial\Omega$ and smooth on $\partial B_\rho(p_i), \partial B_\rho(q_j)$, while $|\Phi_{II}|\le C|\ln\rho|$ on 
$\partial B_\rho(p_i), \partial B_\rho(q_j)$.

The energy of conjugate function $\Phi_{II}$ away from the vortices may then be evaluated as in \cite{BBH2}.  We note that, by means of a rigid rotation by angle $-\beta$, applied to the entire system of antivortices $p_j$, we may obtain $\beta=0$, and that this rotation does not change the value of $\int_{\Omega_\rho} |\nabla \Phi_{II}|^2$.  In particular, this imples that the optimal antivortex configuration is obtained by minimizing the usual Renormalized Energy (defined as in \eqref{RNen}, or expressed in terms of the regular parts of the Green's functions as in \cite{BBH2}) under the constraint  $\beta:=a_1+\cdots+a_I+b_1+\cdots+b_J=0$.
This completes the proof of Theorem~\ref{mainthm} for Problem II.


\paragraph{Problem III.}
For Problem III in the exterior domain $\Omega=\R^2\setminus\overline{\Omega_0}$, let $\Phi_{III}=\Phi_{III}(x;\{p_i,d_i\},\{q_j,\bdeg_j\})$ be any bounded solution of \eqref{PhI} in $\Omega=\R^2\setminus\overline{\Omega_0}$.  Here we obtain the most information, as the solution may be expressed explicitly via Green's functions.  Indeed, for any $p\in\R^2$, $|p|\ge 1$,
$$  G(x, p) = -\ln\left[ \frac{|x- p|\, |x-p^*|}{ |x|^2}\right], 
    \quad p^*:= \frac{p}{ |p|^2},  $$
gives the exterior Neumann Green's function with pole at $p$.  If $|p|>1$, then $G$ solves
$$  -\Delta_x G(x,p) = 2\pi\delta_p(x), \ \text{in $\Omega$,} \quad
     \frac{\partial G}{\partial\nu_x}(x,p)=1, \ \text{for $x\in\Gamma$, $p\in\Omega$,}
$$
whereas if $|p|=1$ (and hence $p^*=p$,) then we have
$$  -\Delta_x G(x,p) = 0, \ \text{in $\Omega$,} \quad
     \frac{\partial G}{\partial\nu_x}(x,p)=1-2\pi\delta_p(x), \ \text{for $x\in\Gamma$, $p\in\Omega$.}
$$
Note that in each case, $G(x,p)$ is bounded outside a neighborhood of $p$, and $G(x,p)\to 0$ as $|x|\to\infty$ for any fixed $|p|\ge 1$.

Proceding as in Lemma~\ref{vt}, we observe that if $p_t=(-t, 0)$ for $t\ge 1$, then $G(x,p_t)$ is even in $x_2$, and $\nabla^\perp G(x,p_t)$ is irrotational in the simply connected domain obtained by deleting a neighborhood of the negative $x_1$-axis from $\Omega$.  In particular, we may write $\nabla^\perp G(x,p_t)=-\nabla\phi_t$ in this domain, and recover a conjugate harmonic map 
$v_t=e^{i\phi_t}$ in $\Omega\setminus\{(-t,0)\}$, satisfying 
$(iv_t,\nabla v_t)=-\nabla^\perp G(x,p_t)$ in $\Omega\setminus\{(-t,0)\}$,
$v_t=e^{i\theta}$ on $\partial B_1(0)\setminus \{(-t,0\}$, and $v_t\to 1$ as $|x|\to\infty$.

For general $p$, $|p|\ge 1$, we again remark that a rotation of the pole $p$ by angle $a$ results in an equivariant rotation on the corresponding $\tilde v_p$, that is $\tilde v_p(z)=e^{i a}v_{|p|}(e^{-i a}z)$.  In particular, if the antivortex location is $p=|p|e^{i(\pi-a)}$, then the limiting value for the conjugate harmonic map will be $v_p(z)\to e^{ia}$ as $|z|\to\infty$.
We may then assemble the harmonic map with vortices $p_1,\dots, p_{\DDD}$,
$v=\prod_{j=1}^{\DDD} \tilde v_{p_j}$, conjugate to the function 
$$  \Phi_{III}(x) = \sum_{j=1}^{\DDD} G(x,p_j) =
\prod_{j=1}^{\DDD} \ln\left[ \frac{ |x|^2}{  |x-p_j|\,|x-p_j^*|}\right],  $$
in the sense that $(iv,\nabla v)=-\nabla^\perp \Phi_{III}(x)$ for
$x\in\Omega\setminus\{p_1,\dots,p_{\mathcal{D}}\}$.  
Writing each antivortex location in the polar form $p_j=|p_j|e^{i(\pi-a_j)}$, we obtain 
$$  v(x)\to e^{i(a_1+\cdots+a_{\mathcal{D}})}, \quad\text{as $|x|\to\infty$}.
$$
Using an equivariant rotation we may ``correct'' this asymptotic value so that $v(x)\to 1$ as $|x|\to\infty$.  The effect of the rotation is to rigidly rotate all of the antivortices by the same angle $-(a_1+\cdots+a_{\mathcal{D}})$, and hence we may restrict our attention to antivortex locations for which the associated angles satisfy
\be\label{angleconstraint}   a_1+\cdots+a_{\mathcal{D}} =0 \mod 2\pi.  
\ee

We may now calculate the energy of limiting antivortex configurations directly using the Green's function representation.  
First, assume each $p_j\in\Omega$, and 
denote by $\Omega_\rho=\Omega\setminus\bigcup_{j=1}^\DDD B_\rho(p_j)$.  Fix vortex locations $p_j$, $j=1,\dots,\mathcal{D}$, and let $R$ be sufficiently large so that $p_j\in B_R(0)$ for all $j=1,\dots,\mathcal{D}$.  Then, we must estimate
\begin{align*}  \int_{\Omega_\rho} |\nabla v|^2\, dx =
\int_{\Omega_\rho} |\nabla\Phi|^2\, dx
 &= \left[\int_{\Omega_\rho\cap B_R(0)}+\int_{\R^2\setminus B_R(0)}\right] 
     |\nabla\Phi|^2\, dx  \\
  &= \int_{\R^2\setminus B_R(0)} |\nabla\Phi|^2\, dx 
    + \left[ \int_{\partial B_R(0)} - \int_{\partial B_1(0)}
       - \sum_{j=1}^{\DDD}\int_{\partial B_\rho(p_j)}\right]
          \Phi\, \partial_\nu \Phi\, ds,
\end{align*}
where in each case the unit normal $\nu$ is chosen positively oriented with respect to each closed curve.

To evaluate the contribution of each integral, we use
$$  \nabla_x G(x,p) = 2\frac{x}{ |x|^2} - \frac{x-p}{|x-p|^2}
   - \frac{x-p^*}{|x-p^*|^2}.  $$
Then, a simple calculation shows that for $\frac12|x|>|p|>1$, 
$$  \left| \frac{x}{ |x|^2} - \frac{x-p}{ |x-p|^2}\right|
\le \frac{1}{ |x|}\,\left| 1-\frac{|x|^2}{|x-p|^2}\right| + \frac{|p|}{|x-p|^2}
\le \frac{4}{ |x|^3}\left| |p|^2 -2x\cdot p\right|+ \frac{4|p|}{ |x|^2} \le \frac{16 |p|}{ |x|^2}.
$$     
In particular, for any $\eps>0$ and any fixed choice of $p_j$, $j=1,\dots,\DDD$, we may choose $R_0$ sufficiently large so that both
$$  \int_{\R^2\setminus B_R(0)} |\nabla\Phi|^2\, dx,
    \left|\int_{\partial B_R(0)} \Phi\partial_\nu\Phi\, ds \right| <\eps,
$$
for all $R\ge R_0$.

For the integral over $\partial B_1(0)$, we recall that $|x-p^*|=|x-p|/|p|$ when $|x|=1$, and $\partial_\nu \Phi=\partial_r\Phi =\DDD$.  Hence,
\begin{align*}
\int_{\partial B_1(0)} \Phi\partial_\nu\Phi\, ds &=
     -\sum_{i=1}^{\DDD} \int_{\partial B_1(0)}
       \DDD\ln \frac{|x-p_i|^2}{ |p_i|}\, ds \\
       &= -2\pi\DDD \sum_{i=1}^{\DDD}\ln |p_i|,
\end{align*}
since $\ln \frac{|x-p_i|^2}{ |p_i|}$ is harmonic in $B_1(0)$.

Next, fix one of the $p_i\in\Omega$, and consider the integral over $\partial B_\rho(p_i)$.  On $\partial B_\rho(p_i)$, we observe that
$$  \partial_\nu\Phi = -\frac{1}{ \rho} + g_i, $$
where $g_i$ is a smooth function in a neighborhood of $p_i$.  Thus, we may write
\begin{align*}
  \int_{\partial B_\rho(p_i)}\Phi\,\partial_\nu\Phi\, ds
 &= \frac{1}{\rho}\int_{\partial B_\rho(p_i)}\left[
   \ln\rho + \sum_{j=1\atop j\neq i}^{\DDD} \ln |x-p_j| 
   + \sum_{j=1}^{|\DDD|}\ln|x-p_j^*|- 2 \DDD\ln|x|
        \right] + o(1) \\
&= 2\pi\left[ \ln\rho + \sum_{j=1\atop j\neq i}^{\DDD} \ln |p_i-p_j| 
   + \sum_{j=1}^{\DDD}\ln|p_i-p_j^*|
      - 2 \DDD\sum_{i=1}^{\DDD}\ln|p_i|
        \right] + o(1).
\end{align*}

Putting these computations together, we obtain an expansion of the energy for fixed vortex locations $p_i\in\Omega$, $i=1,\dots, \DDD$, 
$$ \frac12\int_{\Omega_\rho} |\nabla\Phi|^2\, dx
 = \pi\DDD\ln\frac{1}{\rho} + W(p_1,\dots,p_{\DDD}) + o(1), $$
with Renormalized Energy
\begin{align}\nnn  W(p_1,\dots,p_{\DDD}) &=
 \pi\left[ 3\DDD\sum_{i=1}^{\DDD}\ln|p_i|
         - \sum_{i,j=1}^{\DDD}\ln |p_i-p_j^*|
         - \sum_{i,j=1\atop i\neq j}^{\DDD}\ln |p_i-p_j|\right] \\
         \label{RNIII}
 &=\pi\left[ 2 \sum_{i=1}^{\DDD}\ln|p_i| +
        \sum_{i,j=1}^{\DDD}\ln \frac{|p_i|}{ |p_i-p_j^*|}
        + 2 \sum_{i,j=1\atop i< j}^{\DDD}
           \ln \frac{|p_i||p_j|}{|p_i-p_j|}
 \right]
 \end{align}
We note that 
$$   \frac{|p_i|}{ |p_i-p_j^*|}\ge \frac{|p_i|}{ |p_i|+1}\ge \frac12, $$
and 
$$  \frac{|p_i||p_j|}{|p_i-p_j|} \ge \frac{|p_i||p_j|}{ 2\max\{|p_i|,|p_j|\}}
   \ge \frac12 \min\{|p_i|,|p_j|\}\ge \frac12, $$
and hence we see that $W(p_1,\dots,p_{\DDD})\to +\infty$ whenever:  $|p_i|\to\infty$ for any $i$; or $|p_i|\to 1$ for any $i$; or $|p_i-p_j|\to 0$ for any $i\neq j$.  In particular, $W$ attains a minimum for 
$(p_1,\dots,p_{\DDD})\in\Omega^{\DDD}$ with $p_i\neq p_j$ for all $i\neq j$.

For an arbitrary total degree $\DDD$, the exact location of the vortices of a minimizer may be difficult to determine.  However, in the two cases $\DDD=1,2$ relevant to the application to liquid crystals, we may obtain more information concerning vortex location.
When $\DDD=1$, the form of $W$ is quite simple, and
$$  W(p)=\pi \ln \frac{|p|^3}{ |p-p^*|}.  $$
Taking into account the angle constraint \eqref{angleconstraint} needed to match the boundary condition as $|x|\to\infty$, and writing in complex notation, $p=|p|e^{i\pi}=-|p|$ lies on the left half of the horizontal axis.  Minimizing with respect to $|p|$ yields the optimal vortex location $p=(-2,0)$.

When $\DDD=2$, we write $p_j=t_je^{i(\pi-a_j)}$, $j=1,2$, in complex notation.  Again, to match the condition at infinity, we are constrained to choose $a_2=-a_1=:a$, and hence
\begin{align*} W(p_1,p_2)&= \pi\left[
   6(\ln|p_1|+\ln|p_2|) -2\ln |p_1-p_2|
    - \sum_{i,j=1}^2 \ln|p_i-p_j^*| 
\right] \\
&= \pi\left[
   6(\ln t_1+\ln t_2) - 2 \ln |t_1-t_2e^{-2ia}|
   - \sum_{i,j=1}^2 \ln|t_i-t_je^{2ia}|
   \right].
\end{align*}
We note that each term in $W$ is preserved or decreased by choosing antipodal vortices, $p_2=-p_1$, or $a_2=a_1\pm\pi$.  Given the angle constraint, this implies $a=\pm\frac{\pi}{ 2}$, and so the vortices must lie on opposite halves of the vertical axis, $p_1=(0, t_1)$, $p_2=(0, -t_2)$.  Expressing $W$ for such points,
\begin{align*}
 W((0,t_1), (0,-t_2)) &=
 \ln \left[ \frac{t_1^8 t_2^8}{ 
            (t_1^2-1)(t_2^2-1)(t_1t_2+1)^2}\right]
 \\
 &=:\ln [w(t_1,t_2)],
 \end{align*} 
we may minimize explicitly and obtain the optimal anti-vortex locations, $p_1=(0,\sqrt[4]{2})$ and $p_2=(0,-\sqrt[4]{2})$, as claimed in Theorem~\ref{LQthm}.
\medskip

Next, we assume the vortices lie on the boundary component $\Gamma$: $p_j\in \Gamma$, $i=1,\dots,\DDD$.  Let 
$\Omega_\rho=\Omega\setminus \bigcup_{i=1}^\DDD B_\rho(p_i)$ (as before), and 
$\Gamma^\rho=\partial \left(B_1(0)\cup \bigcup_{i=1}^\DDD B_\rho(p_i)\right)$. We also denote by 
$\tilde\Gamma^\rho=\Gamma\setminus \bigcup_{i=1}^\DDD B_\rho(p_i)$, and $\partial^+ B_\rho(p_i)=\partial B_\rho(x_i)\cap\Omega$.  For vortices $p_i\in\Gamma$ we recall that:
$$  \Phi(x)=\sum_{i=1}^\DDD \ln\frac{ |x|^2}{ |x-p_i|^2},$$
and as above, $\Phi$ is conjugate to the phase of the harmonic map $v$, with $(iv,\nabla v)=-\nabla^\perp\Phi$ away from the vortices.  In this case, we estimate
\begin{align*}
\int_{\Omega_\rho} |\nabla v|^2\, dx =
\int_{\Omega_\rho} |\nabla\Phi|^2\, dx &=
\int_{\R^2\setminus B_R(0)} |\nabla\Phi|^2\, dx
   + \left[\int_{\partial B_R(0)} - \int_{\tilde\Gamma^\rho}
       -\sum_{i=1}^\DDD
         \int_{\partial^+ B_\rho(p_i)}\right] \Phi\, \partial_\nu\Phi\, ds.
\end{align*}
As for the case of interior vortices (above), we may choose $R$ sufficiently large that the integrals over $\R^2\setminus B_R(0)$ and $\partial B_R(0)$ are arbitrarily small, and so it suffices to evaluate the integrals on the inner boundary $\Gamma^\rho=\bigcup_i \partial^+ B_\rho(p_i)\cup \tilde \Gamma^\rho.$

On the circular arcs $\partial^+ B_\rho(p_i)$, we then have
$\partial_\nu\Phi = -\frac{2}{\rho} + g_i(x)$,
where $g_i(x)$ is a smooth function in $B_\rho(p_i)$.  As $\partial^+ B_\rho(p_i)$ differs from a semi-circle $C_\rho^+(p_i)$ by arcs of length of $O(\rho^2)$ as $\rho\to 0$, we have
\begin{align*}
\int_{\partial^+ B_\rho(p_i)} \Phi\partial_\nu\Phi\, ds
&= -\frac{2}{\rho} \int_{\partial^+ B_\rho(p_i)} \Phi\, ds + o(1) \\
&= -\frac{4}{\rho} \sum_{j=1}^\DDD\int_{C_\rho^+(p_i)} 
            \left( \ln |x| - \ln |x-p_j|\right) ds + o(1) \\
&= 4\pi \ln\rho + 4\pi \sum_{j=1\atop j\neq i}^\DDD \ln |p_i-p_j| + o(1).
\end{align*}
On the arcs making up $\tilde \Gamma^\rho\subset \partial B_1(0)$, we have $\partial_\nu \Phi=\DDD$.  We also note that for any $p\in S^1$, 
$\ln|x-p|^2\in L^1(\Gamma)$, and 
$$  \int_{\Gamma} \ln |x-p|^2\, ds = c_0 $$
is a constant, independent of $p\in S^1$.
Therefore, we may evaluate
\begin{align*}
\int_{\tilde\Gamma^\rho} \Phi\, \partial_\nu\Phi\, ds
&= \DDD \int_{\tilde\Gamma^\rho} \Phi\, ds
= \DDD \int_{\Gamma} \Phi\, ds + o(1) \\
&= -\DDD \sum_{i=1}^\DDD \int_\Gamma \ln |x-p_i|^2\, ds +o(1)
= -\DDD^2 c_0 + o(1).
\end{align*}

Putting these computations together, we  obtain
\begin{align*}
\frac12\int_{\Omega_\rho} |\nabla v|^2\, dx &=
\frac12\int_{\Omega_\rho} |\nabla\Phi|^2\, dx  \\
&= 2\pi \DDD \ln \frac{1}{\rho} 
- 2\pi \sum_{j=1\atop j\neq i}^\DDD \ln |p_i-p_j| 
+ \frac{\DDD^2}{ 2}c_0 + o(1).
\end{align*}
Thus, the Renormalized Energy for vortices lying on the circle $\Gamma$ is
$$  W_\Gamma(p_1,\dots,p_\DDD) = -2\pi \sum_{j=1\atop j\neq i}^\DDD \ln |p_i-p_j| 
+ \frac{\DDD^2}{ 2}c_0,
$$
and is minimized by vortices which are evenly distributed over the cirlce $\Gamma$.  As for the case of interior vortices, the asymptotic condition $v\to 1$ as $|x|\to\infty$ imposes the constraint \eqref{angleconstraint} on the polar angles of the $p_i$, which removes the degeneracy of the minimizing configuration due to rotational invariance.  In particular, in case $\DDD=1$, the single anti-vortex must be located on the left side of the horizontal axis, $p=(-1,0)$, and for $\DDD=2$, the two anti-vortices lie on opposite sides of the vertical axis, $p_1=(0,1)=-p_2$.  This completes the proof of Theorem~\ref{LQthm}.

\paragraph{Acknowledgements}
The authors wish to thank Vincent Millot for introducing them to the Q-tensor model, and for valuable discussions on its connection to Ginzburg--Landau and micromagnetics.  The first two authors are supported through NSERC (Canada) Discovery Grants.


\bibliographystyle{amsalpha}
\addcontentsline{toc}{section}{References}
\bibliography{refs}


\end{document}